\documentclass{article}
\usepackage[a4paper, margin=1in]{geometry} 
\usepackage{graphicx} 
\usepackage{forest,tikz,amsfonts,amsmath,amssymb,amsthm}

\usepackage[a4paper, margin=1in]{geometry}

\usepackage{subcaption} 

\usetikzlibrary{shapes.geometric, arrows}
\usetikzlibrary{arrows.meta, positioning}

\usepackage{tikz}
\usetikzlibrary{calc}
\usetikzlibrary{backgrounds}

\usepackage[table]{xcolor}

\usepackage{ytableau}
\usepackage{multicol}

\usepackage{algorithm}
\usepackage{algpseudocode}
\usepackage{float}

\usepackage{dsfont}

\usepackage{enumitem}

\usepackage{dsfont}

\usepackage{hyperref}
\hypersetup{
    colorlinks=true,
    linkcolor=blue,
    filecolor=magenta,      
    urlcolor=cyan,
    pdftitle={Overleaf Example},
    pdfpagemode=FullScreen,
    }

\usepackage{cleveref}

\def\V#1{{\mathbf #1}}
\def\floor#1{{\lfloor #1 \rfloor}}

\def\Unif{\operatorname{Unif}}

\def\P{\mathbb{P}}
\def\E{\mathbb{E}}
\def\B{\operatorname{B}}

\def\pmod{\operatorname{mod}}

\DeclareMathOperator*{\argmax}{arg{\,}max}

\newcommand{\gr}{\textsc{greedy}}

\newtheorem{lemma}{Lemma}
\newtheorem{proposition}{Proposition}
\newtheorem{theorem}{Theorem}
\newtheorem{corollary}{Corollary}
\newtheorem{conjecture}{Conjecture}
\newtheorem{remark}{Remark}

\allowdisplaybreaks

\title{Burning rooted graph products}
\author{John Peca-Medlin\thanks{Department of Mathematics, University of California, San Diego, \href{mailto:jpecamedlin@ucsd.edu}{jpecamedlin@ucsd.edu}}}
\date{}

\begin{document}
\maketitle

\begin{abstract}
    The burning number $b(G)$ of a graph $G$ is the minimum number of rounds required to burn all vertices when, at each discrete step, existing fires spread to neighboring vertices and one new fire may be ignited at an unburned vertex. This parameter measures the speed of influence propagation in a network and has been studied as a model for information diffusion and resource allocation in distributed systems. A central open problem, the Burning Number Conjecture (BNC), asserts that every graph on $n$ vertices can be burned in at most $\lceil \sqrt n\rceil$ rounds, a bound known to be sharp for paths and verified for several structured families of trees. We investigate rooted graph products, focusing on comb graphs obtained by attaching a path (a ``tooth'') to each vertex of a  path (the ``spine''). Unlike classical symmetric graph products, rooted products introduce hierarchical bottlenecks: communication between local subnetworks must pass through designated root vertices, providing a natural model for hub-and-spoke or chain-of-command architectures. We prove that the BNC holds for all comb graphs and determine the precise asymptotic order of their burning number in every parameter regime, including exact formulas in the spine-dominant case that generalize the known formula for paths. Our approach is constructive, based on an explicit greedy algorithm that is optimal or near-optimal depending on the regime.
\end{abstract}

\section{Introduction}

The burning process is a discrete-time process on a graph $G$. 
At time $t=0$, all vertices are unburned. 
At each subsequent time step $t+1$, every existing fire spreads fire to all of its neighbors, while in addition a new fire is started at a single unburned vertex, if available. 
The \textit{burning number} of $G$, denoted $b(G)$, is the minimum number of steps required to burn all vertices of $G$.

An equivalent characterization is in terms of sphere packing on a graph.  Let $d(u,v)$ denote the usual graph distance, and for $v \in V(G)$ and $r \ge 0$, let $B(v,r) = \{ u \in V(G) : d(u,v) \le r \}$. Then
\begin{equation*}
b(G) = \min \left\{ k : \exists v_1,\ldots,v_k \in V(G) \text{ such that }
V(G) = \bigcup_{i=1}^k B(v_i,k-i) \right\}.
\end{equation*}
Here $B(v_i,k-i)$ consists precisely of the vertices that can be burned by the fire started at time $i$ by time $k$. 

Establishing the burning number for a graph is NP-complete, including also for trees and other simple families of graphs \cite{BessyBonatoJanssenRautenbachRoshanbin2017}.  That said, several basic graph families admit explicit formulas. 
For the path $P_n$ with $n$ vertices, one has
\[
b(P_n) = \lceil \sqrt{n} \rceil.
\]
To see this:  since $|B(v,r) \cap P_n| \le 2r+1$, then using $k$ balls of radii $k-1,k-2,\dots,0$ covers at most
\[
\sum_{i=1}^k (2(k-i)+1) = \sum_{i=0}^{k-1} (2i+1) = k^2
\]
vertices. So $k^2 \ge n$ is necessary and sufficient. See \Cref{fig: P_16} for an optimal burning sequence of $P_{16}$.

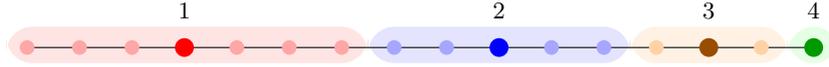
\begin{figure}[t]
\centering
\begin{tikzpicture}[scale=0.69]

\begin{scope}[xshift=0cm] 

\def\n{16}
\def\H{1}

\foreach \j in {1,...,\n} {
    \node[circle,fill=black,inner sep=1.5pt] (s\j) at (\j,-1) {};
}
\foreach \j in {1,...,15} {
    \draw (s\j) -- (s\the\numexpr\j+1\relax);
}

\begin{scope}[on background layer]

\fill[red!30,opacity=0.35,rounded corners=10pt] 
    (0.6,-.6) -- (7.5,-0.6) -- (7.5,-1.3) -- (0.6,-1.3) -- cycle;

\fill[blue!30,opacity=0.35,rounded corners=10pt]
    (7.5,-.6) -- (12.5,-.6) -- (12.5,-1.3) -- (7.5,-1.3) -- cycle;

\fill[orange!30,opacity=0.35,rounded corners=10pt]
    (12.5,-.6) -- (15.5,-.6) -- (15.5,-1.3) -- (12.5,-1.3) -- cycle;

\fill[green!30,opacity=0.35,rounded corners=10pt]
    (15.5,-.6) -- (16.5,-.6) -- (16.5,-1.3) -- (15.5,-1.3) -- cycle;

\end{scope}

\foreach \j in {1,...,7}{\fill[red!35] (\j,-1) circle (4pt); }
\node[circle,fill=red,inner sep=2.5pt] at (4,-1) {};

\foreach \j in {8,...,12}{ \fill[blue!35] (\j,-1) circle (4pt);}
\node[circle,fill=blue,inner sep=2.5pt] at (10,-1) {};

\foreach \j in {13,...,15}{  \fill[orange!35] (\j,-1) circle (4pt);} 
\node[circle,fill=orange!60!black,inner sep=2.5pt] at (14,-1) {};

\node[circle,fill=green!60!black,inner sep=2.5pt] at (16,-1) {};


\node at (4,-0.3) {\small 1};
\node at (10,-0.3) {\small 2};
\node at (14,-0.3) {\small 3};
\node at (16,0-.3) {\small 4};

\end{scope}

\end{tikzpicture}
\caption{Optimal burning sequence using $4$ fires for $P_{16}$.}
\label{fig: P_16}
\end{figure}

At the opposite extreme, for the complete graph $K_n$ with $n \ge 2$, one has $B(v,1)=V(K_n)$ for any vertex $v$, and hence $b(K_n)=2$. 
For perfect binary trees with $n=2^h-1$ vertices and height $h=\log_2(n+1)$, the optimal strategy is to ignite the root and allow the fire to propagate, yielding $b(T)=h+1=\Theta(\log n)$. Optimality is established by showing any other burning sequence cannot burn enough leaves in a shorter time span (see \cite{das2023burning}).

A central open problem in the area is the Burning Number Conjecture (BNC), introduced in \cite{bonato2016burn}.

\begin{conjecture}[BNC]
If $G$ is a connected graph on $n$ vertices, then
\[
b(G) \le \lceil \sqrt{n} \rceil.
\]
\end{conjecture}

\noindent In essence, the conjecture asserts that the path is extremal among connected graphs. 

The conjecture has been verified for several graph classes, including paths, cycles, Hamiltonian graphs \cite{bonato2016burn}, caterpillars \cite{liu2020burning,hiller2020burning}, spiders \cite{bonato2019bounds,das2018burning}, and trees without degree-two vertices \cite{murakami2024burning}. 
General progress toward the conjecture has led to increasingly sharp upper bounds (e.g., \cite{land2016upper,bessy2018bounds}). 
The furthest (nonasymptotic) advancement is due to Bastide et al.\ \cite{BastideBonamyBonatoCharbitKamaliPierronRabie2023}, who proved
\[
b(G) \le \left\lceil \sqrt{\frac{4}{3}n} \right\rceil + 1.
\]
Asymptotically, Norin and Turcotte \cite{norin2024burning} have resolved the BNC, showing that
\[
b(G) \le (1+o(1))\sqrt{n}.
\]

Beyond extremal bounds, burning numbers have also been studied for random graph models. 
For Erd\H{o}s--R\'enyi graphs $G(n,p)$ with sufficiently large $p=p_n$, the burning number is $O(1)$ with high probability \cite{mitsche2017burning}. 
For critical Galton--Watson trees conditioned to have $n$ vertices with finite variance offspring distribution, it has been shown that $b(G) = \Theta(n^{1/3})$ with high probability \cite{devroye2025burning}. 
Other recent directions include burning growing graphs and lattices \cite{barrett2026achievable}, as well as randomly burning graphs \cite{blanc2025random}.

Graph products provide a natural framework for constructing larger networks from simpler components. 
Mitsche, Pra{\l}at, and Pawe{\l} \cite{mitsche2018burning} studied burning numbers of symmetric graph products and proved that for connected graphs $G$ and $H$,
\begin{equation}\label{eq: product inequalities}
\max\{b(G),b(H)\}
\le
b(G \boxtimes H)
\le
b(G \,\square\, H)
\le
\min\{ b(G)+\mathrm{rad}(H),\, b(H)+\mathrm{rad}(G) \},
\end{equation}
where $G \boxtimes H$ denotes the strong product, $G \,\square\, H$ the Cartesian product, and $\operatorname{rad}(G)$ denotes the radius of a graph $G$, $\mathrm{rad}(G) = \min_{v \in V(G)} \operatorname{ecc}(v)$, where the eccentricity of a vertex $v$ is
$\operatorname{ecc}(v) = \max_{u \in V(G)} d(u,v)$. These symmetric products model diffusion in relatively homogeneous networks: propagation occurs along multiple independent directions, reflecting grid-like communication structures in which no single vertex plays a central routing role.
Similarly, the lexicographic product (also studied in \cite{mitsche2018burning,bonato2021improved}) introduces dense inter-layer connectivity, enabling rapid cross-layer dissemination once adjacency in the base graph is established. 
In all these cases, communication is symmetric or highly interconnected, and no structural bottlenecks significantly constrain the spread of information.

The rooted product captures a fundamentally different architecture. 
Given graphs $G$ and $H$ with a distinguished root vertex $r \in V(H)$, the rooted product $G \circ H$ is obtained by taking one copy of $G$ and, for each vertex $v \in V(G)$, attaching a copy of $H$ by identifying $v$ with the root $r$ of that copy. 
Each vertex of $G$ therefore serves as a gateway to a local copy of $H$. 
In $G \circ H$, every copy of $H$ connects to the remainder of the graph through a single vertex, creating unavoidable communication bottlenecks and producing a hierarchical network consisting of a backbone (the graph $G$) together with local subnetworks (copies of $H$). 
Inter-branch communication must pass through designated gateway vertices, in contrast to the parallel diffusion present in strong or Cartesian products. 
Rooted products thus provide a natural deterministic model for constrained hierarchical systems, such as chains of command or hub-and-spoke networks, where dissemination is mediated through structured control nodes.

In this paper, we study the burning number of rooted products of paths.

\subsection{Main results}
The first contribution is for general rooted graph products. We establish an immediate relationship of the rooted graph product back to the Cartesian product of graphs:

\begin{theorem}\label{thm: rooted ineq}
    Let $G$ be a graph and $H$ a rooted graph with root $r_H \in V(H)$. Then
    \[
    b(G \, \square \, H) \le b(G \circ H) \le b(G) + \operatorname{ecc}(r_H).
    \]
\end{theorem}

This result is direct enough, we include a proof here. 

\begin{proof}
    Since $G \circ H$ is an isometric subgraph in $G \, \square \, H$, then $b(G \, \square \, H) \le b(G \circ H)$ (see \cite{bonato2016burn}). We can define a burning sequence on $G \circ H$ by first using a minimal burning sequence with $b(G)$ fires on $G$ and then waiting $\operatorname{ecc}(r_H)$ additional rounds to let each fire fully extend to each copy of $H$. This establishes $b(G \circ H) \le b(G) + \operatorname{ecc}(r_H)$.
\end{proof}

\Cref{thm: rooted ineq} allows rooted products to be compared directly with the Cartesian and strong products in \eqref{eq: product inequalities}. 
In particular, if the roots $r_G \in V(G)$ and $r_H \in V(H)$ are chosen to be centers of their respective graphs, so that $\operatorname{ecc}(r_G)=\operatorname{rad}(G)$ and $\operatorname{ecc}(r_H)=\operatorname{rad}(H)$, then rooted products fit in between the final two bounds in \eqref{eq: product inequalities}.
\begin{corollary}
Let $G$ and $H$ be graphs, and let $r_G \in V(G)$ and $r_H \in V(H)$ be centers of each graph that are identified as root vertices. Then
\[
    b(G \, \square \, H) \le \min\{b(G \circ H),b(H \circ G)\} \le \min\{b(G) + \operatorname{rad}(H),b(H) + \operatorname{rad}(G)\}.
\] 
\end{corollary}


We now specialize to the case of products of paths. Our primary focus is the rooted product $P_n \circ P_m$, where an end-vertex of $P_m$ is designated as the root. This construction yields the \textit{comb graph}
\[
C_{n,m} = P_n \circ P_m,
\]
consisting of $nm$ vertices. We refer to the copy of $P_n$ as the \emph{spine} and to each attached copy of $P_m$ as a \emph{tooth}. The resulting spine–tooth geometry imposes a strict hierarchy: interaction between teeth must occur along the spine, while propagation within each tooth remains local.

This configuration models systems consisting of parallel linear units coupled through a single backbone. When the spine is long relative to the teeth (spine-dominant), global coordination along the backbone determines the propagation time; when the teeth are longer (tooth-dominant), local diffusion within branches becomes the bottleneck. The regime $n \approx m$ interpolates between these coordination- and locality-dominated behaviors.

Our main technical contribution is the following:
\begin{theorem}\label{thm: bnc}
    The BNC is true for comb graphs.
\end{theorem}

We establish this using an explicit greedy algorithm introduced in \Cref{sec: greedy} for generating a burning sequence on $C_{n,m}$. Let $T_\gr$ denote the minimal successful run time for this greedy algorithm. We show then:
    \[
    b(C_{n,m}) \le T_\gr \le \lceil \sqrt{nm}\rceil.
    \]

\begin{figure}[t]
    \centering
    \includegraphics[height=5.75cm]{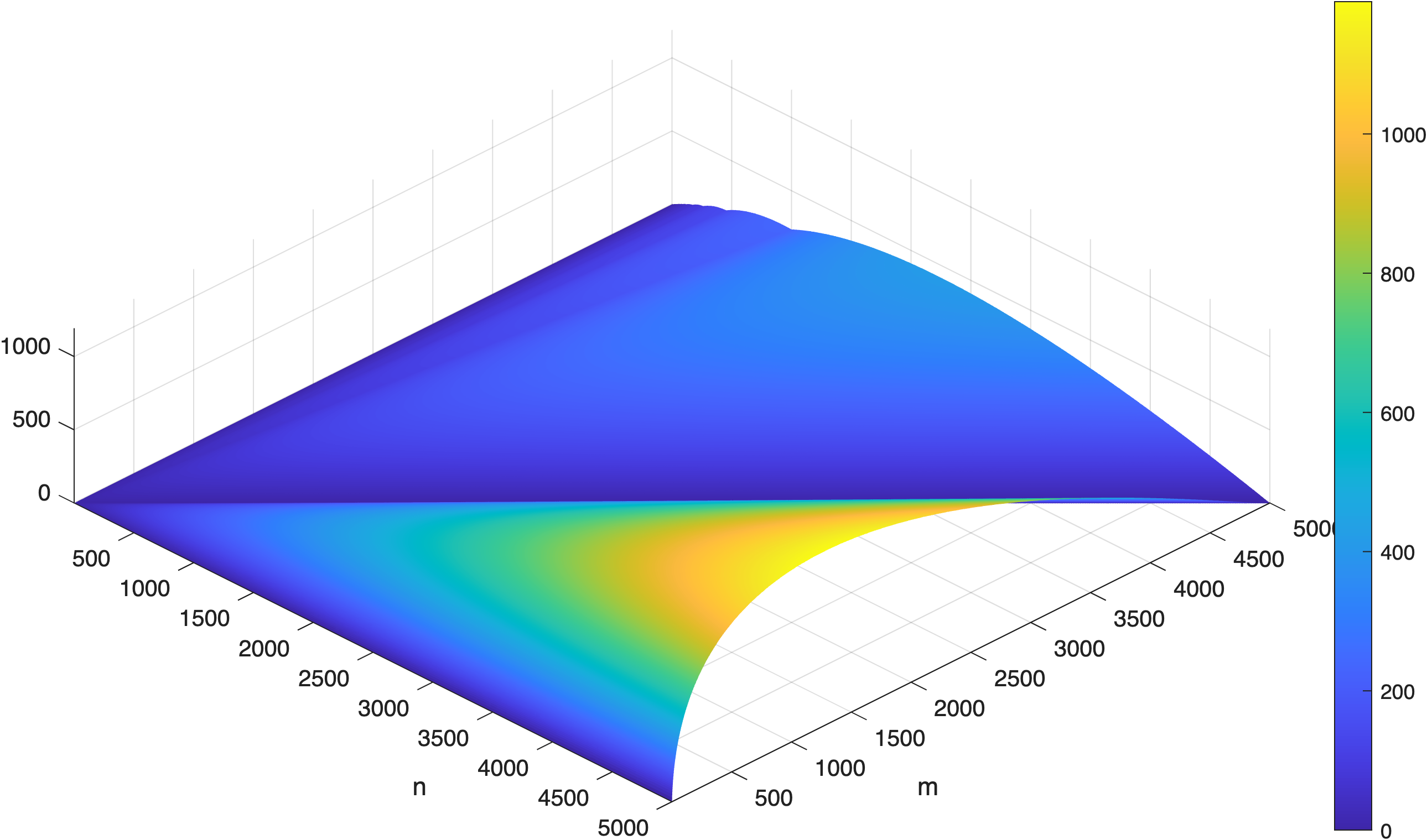}
    \includegraphics[height=5.75cm]{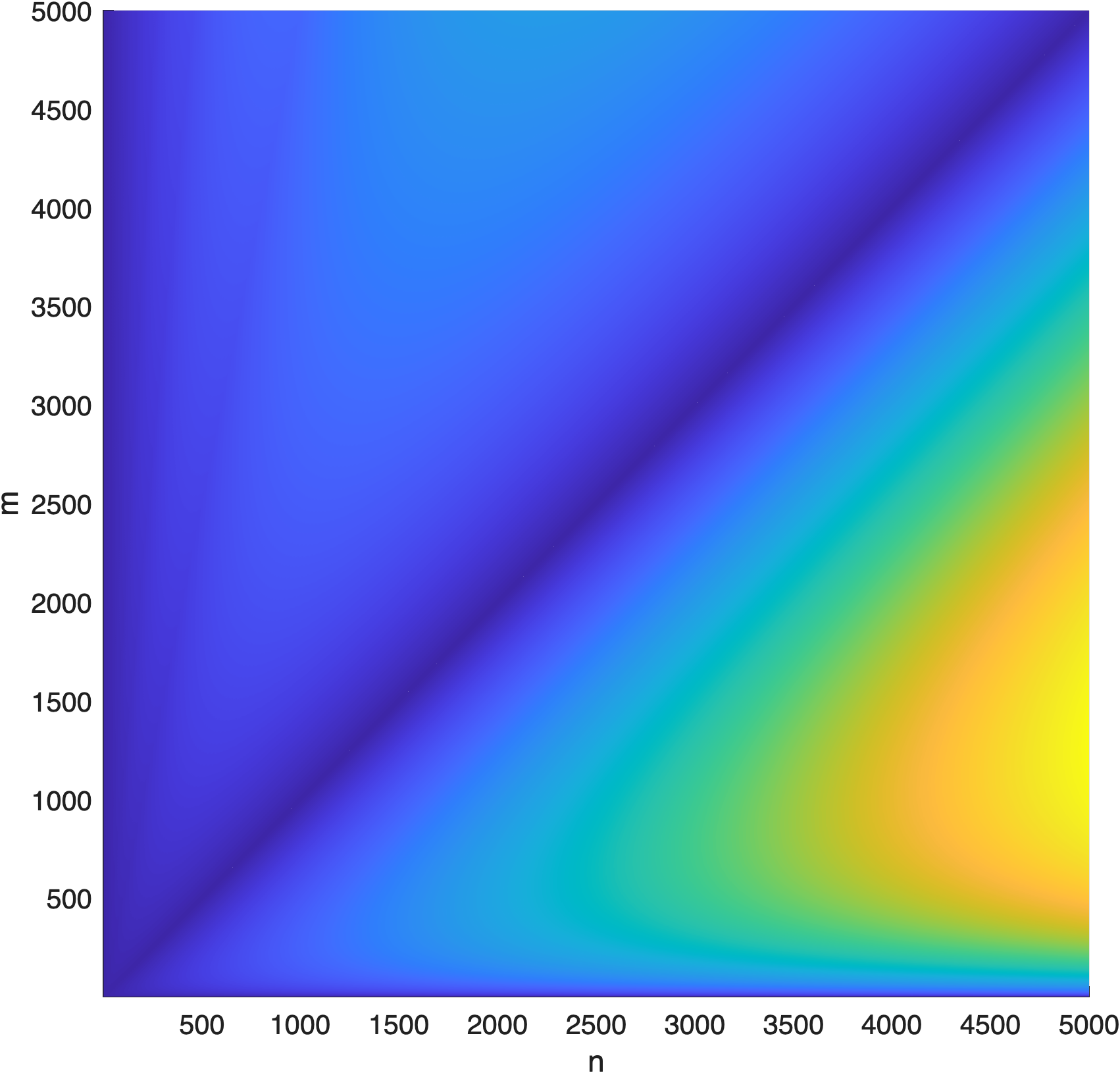}
    \caption{Plot of the gap $\lceil \sqrt{nm} \rceil - T_{\gr}$ for $1 \le m,n \le 5{,}000$ for burning $C_{n,m}$. For $n \le m$, the maximal gap of 417 is achieved by $m = 5{,}000$ and $2{,}182 \le n \le 2{,}266$. For $ n \ge m$, the maximal gap of
$1{,}190$ is achieved by 42 pairs $(n,m)$ where $n= 5{,}000$ and $m$ is in the range between $1{,}280$ to $1{,}321$.}
    \label{fig:b_comparison}
\end{figure}

Using this algorithm, we distinguish the burning behavior in the spine-dominant ($n \ge m$) and tooth-dominant ($n \le m$) regimes. Beyond verifying the BNC scale, we quantify the gap between the conjectured upper bound $\lceil \sqrt{nm} \rceil$ and the greedy completion time $T_\gr$. \Cref{fig:b_comparison} visualizes this gap for all $1 \le n,m \le 5{,}000$, both as a surface plot and as a heat map. The BNC bound is tight along the boundaries $n=1$, $m=1$, and $n=m$. The first two cases reflect the fact that $C_{n,m}$ reduces to a path when either parameter equals 1, while the diagonal case $n=m$ is established in \Cref{l: n=m}, where we show $b(C_{n,n})=n$.

For the spine-dominant case, we provide the exact burning number by showing $T_\gr$ is optimal:

\begin{proposition}\label{prop: spine-dominant}
    If $n \ge m$, then
    \[
    b(C_{n,m}) = T_\gr = m-1 + \lceil \sqrt{n - m + 1}\rceil.
    \]
\end{proposition}

This fully generalizes the path formula $b(P_n)=\lceil \sqrt n\rceil$, recovered in the case $m=1$. Optimality follows by identifying $T_\gr$ with the minimal greedy completion time and showing that any optimal burning sequence can be converted into one produced by the greedy strategy.

Next, in the tooth-dominant regime, we establish the greedy algorithm always succeeds using $\lceil \sqrt{nm}\rceil$ fires. Moreover, we show $T_\gr$ is within $n/2$ of this upper bound and captures the true scaling of the burning number:

\begin{proposition}\label{prop: tooth-dominant}
    If $n \le m$, then 
    \[
    \sqrt{nm/2} \le b(C_{n,m}) \le T_\gr = \lceil \sqrt{nm}\rceil - d_{n,m}, \qquad d_{n,m} \in [0,n/2].
    \]
\end{proposition}

The two regimes exhibit distinct scaling laws: in the spine-dominant case,
\[
b(C_{n,m}) = \Theta\!\big(\max\{\sqrt n,\, m\}\big),
\]
so influence propagates relatively quickly along a long backbone supporting many small attached communities. In contrast, when the teeth are long and the spine short, the burning number grows on the order of $\sqrt{nm}$, placing the spread near the slowest scale permitted for a graph of this size (assuming the BNC). The distinction lies in how influence is distributed across the comb architecture: many small branches along an extended spine promote efficient dissemination, while a few large branches attached to a short spine significantly delay it. Between these extremes lies a transitional regime in which the comb architecture reproduces the same leading-order scaling as rectangular grids.

More broadly, the scaling analysis of $C_{n,m}$ extends naturally to random rooted graph products of paths and enables comparison with classical Cartesian and strong products. For instance, let $U_k \sim \Unif\{0,1,\dots,k\}$ and define $X_k = 2^{U_k}$ and $Y_k = 2^{k-U_k}$. The resulting random comb $C_{X_k,Y_k}$ has $2^k$ vertices, with geometry ranging from spine-dominant to tooth-dominant according to $U_k$. Combining our scaling results with known asymptotics for Cartesian and strong products of paths \cite{mitsche2018burning,bonato2021improved}, we obtain the following distributional limit. An abbreviated form of \Cref{thm: random comb - long} is stated below.

\begin{theorem}\label{thm: random comb_short}
Let $U_k \sim \Unif\{0,1,\ldots,k\}$ with $X_k = 2^{U_k}$ and $Y_k = 2^{k-U_k}$. Then
\[
\frac1k \log_2 b(C_{X_k,Y_k}) \xrightarrow[k\to\infty]{d} f(U),
\]
and analogous limits hold for $P_{X_k}\square P_{Y_k}$ and $P_{X_k}\boxtimes P_{Y_k}$, where $U \sim \Unif(0,1)$ and $f$ is given explicitly in \Cref{thm: random comb - long}. Moreover, $f(U)$ has full support on $[1/3,1/2]$.
\end{theorem}

Unlike previously studied random models, where suitably scaled burning numbers converge to deterministic constants (see \cite{devroye2025burning,mitsche2017burning}), this construction yields a genuinely non-degenerate limiting distribution.

We also examine comb graphs as spanning trees of rectangular lattices. Although combs recover the correct burning scale in certain aspect-ratio regimes, they do not capture the full two-dimensional expansion of the lattice when the geometry becomes sufficiently skewed. This motivates the study of alternative recursive spanning trees, including butterfly trees \cite{PZ25,peca2025horton}.


\subsection{Outline}

\Cref{sec: greedy} introduces the family of greedy algorithms used to constructively approach \Cref{thm: bnc}, which is proved in \Cref{sec: bnc comb}. In \Cref{sec: balanced}, we establish exact burning numbers for combs in the balanced and nearly balanced regimes, where $|n-m| \le 5$. These results lead to \Cref{sec: spine dom}, which determines the exact burning number in the spine-dominant regime. \Cref{sec: tooth dom} then addresses the tooth-dominant regime, proving the BNC and showing that the greedy algorithm is nearly optimal (see \Cref{prop: approx}). Finally, \Cref{sec: random} develops the random graph product constructions, including the full statement and proof of \Cref{thm: random comb - long}, while \Cref{sec: trees} connects our results to other recursive tree constructions.

\section{Greedy algorithm}\label{sec: greedy}

In this section we introduce the greedy algorithm underlying the constructive proof of \Cref{thm: bnc}.  
Given $n,m$ and a prescribed number of rounds $T$, the algorithm constructs a candidate burning sequence for $C_{n,m}$ in $T$ steps and returns both the sequence and a success indicator specifying whether $C_{n,m}$ is burned within $T$ rounds.

The construction proceeds in two phases, reflecting the spine-dominant and tooth-dominant regimes.

\begin{itemize}
    \item If $n \ge m$ (spine-dominant) and $T \ge m$, the first $T-m+1$ fires are placed consecutively along the spine, starting from one end. By time $T$, these fires burn the entire spine and all but at most $m-1$ teeth. The remaining unburned vertices form a path forest consisting of at most $m-1$ vertical segments.

    \item If $n \le m$ (tooth-dominant) and $T \ge n$, the first fire is placed at the top-left corner of $C_{n,m}$. By time $T$ this fire burns the entire spine, and the remaining unburned vertices form a path forest with at most $n$ components.
\end{itemize}

After this reduction, the problem becomes that of burning a path forest, and we apply the standard greedy strategy for path forests studied in \cite{bonato2019bounds}. A necessary condition for success is $T \ge \min\{n,m\}$ (see \Cref{l: n=m}), ensuring that the reduction phase indeed produces a path forest prior to the final greedy completion.

The procedure \textsc{GreedyComb}$(T;n,m;S)$ (see \Cref{alg: comb}) formalizes this construction. The parameter $S$ specifies the horizontal offset of the final spine fire in the reduction phase. After reduction, the uncovered vertices form a path forest, to which we apply the greedy path-forest algorithm \textsc{GreedyPathForest} (see \Cref{alg: forest}).

The algorithm outputs both a candidate burning sequence and a success flag. We define
\[
T_\gr^{(S)}(n,m)
=
\min\left\{T \ge 1 :
\textsc{GreedyComb}(T;n,m;S) \text{ succeeds}\right\}.
\]
Our primary focus is
\[
T_\gr = T_\gr^{(1)},
\]
corresponding to placing the final reduction-phase fire at the leftmost remaining tooth. When $n \ge m$ and $T \ge m$, this fire occurs in round $T-m+1$; when $n \le m$, it is the initial fire.

This explicit algorithm provides the foundation for the analysis in the subsequent sections.

\begin{algorithm}[ht]
\caption{\textsc{GreedyComb}$(T;n,m; S)$}
\label{alg:greedy_comb}
\begin{algorithmic}[1]
\Require $T,n,m,S \in \mathbb N^+$
\Ensure $\mathrm{Success}, \mathrm{BurningSequence}$

\State Construct integer lattice of vertices $(i,j)$, $1 \le i \le n$, $1 \le j \le m$, 
\Statex \hspace{1em} where $(i,j)$ denotes the $i$th vertex of tooth $j$ and $(1,j)$ the spine vertex
\State Set all vertices with $\mathrm{status}=0$ (i.e., unburned)
\State $\mathrm{Success} \gets 0$
\State $\mathrm{BurningSequence} = \varnothing$
\State $k \gets 1$


\If{$n \ge m$} \Comment{{Initial reduction phase: Spine-dominant regime}}
    \While{$k \le T-m$}
        \State Let $v_k$ be the spine vertex  distance $T-k$ from the left-most unburned leaf
        \Statex \hspace{4em} or the final spine vertex if no such spine vertex remains
        \State Set all vertices within distance $T-k$ of $v_k$ to $\mathrm{status}=1$
        \State $\mathrm{BurningSequence} \gets \mathrm{BurningSequence}.\texttt{append}(v_k)$
        \State $k \gets k+1$
    \EndWhile

    \If{All vertices have $\mathrm{status}=1$}
        \State $\mathrm{Success} = 1$
        \Return
    \EndIf
\EndIf


\If{$k \le T$} \Comment{{Final initial fire placement, depending on $S$}}
    \State Let $v_k$ be the spine vertex above the $S$ left-most unburned leaf 
    \Statex \hspace{2em} or the terminal spine vertex if fewer than $S$ unburned leaves remain
    \State Set all vertices within distance $T-k$ of $v_k$ to $\mathrm{status}=1$
    \State $\mathrm{BurningSequence} \gets \mathrm{BurningSequence}.\texttt{append}(v_k)$
    \State $k \gets k+1$
    \medskip
    
    \If{All vertices have $\mathrm{status}=1$}
        \State $\mathrm{Success} = 1$
        \Return
    \EndIf
\EndIf

\If{$k \le T$}\Comment{{Greedy completion on remaining path forest}}

    \State Let $L_1,\ldots,L_s$ denote the path forest of remaining unburned vertical line segments, 
    \Statex \hspace{2em} where each $L_i = (v^i_1,\ldots,v^i_{\ell_i})$ ordered from left to right, is ordered from top to bottom
    \State $(\mathrm{Success}',\mathrm{BurningSequence}') = $ \textsc{GreedyPathForest}$(T-k+1; L_1,L_2,\ldots,L_s)$
    \State $\mathrm{Success} \gets \mathrm{Success}'$
    \State $\mathrm{BurningSequence} \gets \mathrm{BurningSequence}.\texttt{append}(\mathrm{BurningSequence}')$
\EndIf

\State \Return
\end{algorithmic}
\label{alg: comb}
\end{algorithm}

\begin{algorithm}[ht]
\caption{\textsc{GreedyPathForest}$(T; L_1,\ldots,L_s)$}
\label{alg:greedy_path_forest}
\begin{algorithmic}[1]
\Require $T \in \mathbb N$ and ordered path components $L_1,\ldots,L_s$,
\Statex \hspace{2em} where each $L_i = (v^i_1,\ldots,v^i_{\ell_i})$ is given from top to bottom
\Ensure $\mathrm{Success}, \mathrm{BurningSequence}$

\State $\mathrm{Success} \gets 0$
\State $\mathrm{BurningSequence} \gets \varnothing$
\State $k \gets 1$

\While{$k \le T$}

    \If{$s = 0$}
        \State $\mathrm{Success} \gets 1$
        \Return
    \EndIf

    \State Let $\displaystyle i^\ast \gets \argmax_{1 \le i \le s} \texttt{length}(L_i)$ \Comment{Select longest component, using $\texttt{length}(P_\ell) = \ell$}
    \State $P \gets L_{i^\ast}$


    \If{$\texttt{length}(P) \ge T-k+1$} \Comment{{Select fire location}}
        \State Let $v_k$ be the vertex of $P$ at distance $T-k$ from its top endpoint (i.e., $v^{i^\ast}_{\,T-k+1}$)
    \Else
        \State Let $v_k$ be the last vertex of $P$
    \EndIf


    \State $B \gets \{v \in P : d(v,v_k) \le T-k\}$ \Comment{{Burn radius and update component}}
    \State $P \gets P.\texttt{remove}(B)$

    \If{$\texttt{length}(P) = 0$}
        \State Remove $L_{i^\ast}$ from $(L_1,\ldots,L_s)$
        \State $s \gets s-1$
    \Else
        \State $L_{i^\ast} \gets P$
    \EndIf

    \State $\mathrm{BurningSequence} \gets \mathrm{BurningSequence}.\texttt{append}(v_k)$
    \State $k \gets k+1$

\EndWhile

\If{$s=0$}
    \State $\mathrm{Success} \gets 1$
\EndIf

\State \Return
\end{algorithmic}
\label{alg: forest}
\end{algorithm}

\section{Burning number for \texorpdfstring{\boldmath$C_{n,m}$}{C_nm}}\label{sec: bnc comb}

We now prove \Cref{thm: bnc} for comb graphs $C_{n,m}$ via \textsc{GreedyComb}. After handling the balanced case $|n-m|$ small (in \Cref{sec: balanced}), we treat the spine- and tooth-dominant regimes separately (in \Cref{sec: spine dom,sec: tooth dom}).

\subsection{\texorpdfstring{\boldmath$|n - m| \le 5$}{n ~ m}: balanced combs}
\label{sec: balanced}

We begin with the regime in which $n$ and $m$ are close. These balanced cases serve as the base configurations for the inductive arguments in the spine-dominant and tooth-dominant regimes. In particular, we will require exact values when $n=m$, and when $m=n+1$, which provide the initial thresholds for the later comparisons with $T_\gr$.

To obtain sharp lower bounds in these cases, it is convenient to work with the uniform-radius covering parameter
\begin{equation*}
    \hat b(G) = \min\{r : \hat b_{r-1}(G) \le r\}, 
    \quad 
    \hat b_r(G) = \min \left\{ k : \exists v_1,\ldots,v_k \in V(G) \mbox{ such that }
    V(G) = \bigcup_{i=1}^k B(v_i,r) \right\}.
\end{equation*}
Thus $\hat b_r(G)$ is the minimum number of radius-$r$ balls required to cover $G$, and $\hat b(G)$ is the smallest $r$ for which $r$ such balls suffice. Equivalently, $\hat b(G)$ is the minimal burning time if all fires are ignited simultaneously and expand with the same radius. Similarly, this can be equated to a sphere packing problem on the graph, while now using the same radius for each sphere.

We note a few observations relating the parameters $\hat b(G)$ and $b(G)$. Any burning sequence yields a covering using uniform radii balls as expanding the burning radii of each subsequent fire still covers the graph. Conversely, placing the first $\hat b(G)$ fires for a standard burning sequence at the locations of a successful uniform radius covering will then consume the graph after waiting an additional $\hat b(G)$ rounds. Hence
\begin{equation*}
    \hat b(G) \le b(G) \le 2\hat b(G).
\end{equation*}
It follows $b(G) = \Theta(\hat b(G))$. This parameter was used in \cite{devroye2025burning} to determine sharp burning asymptotics for conditioned critical Galton--Watson trees.

We will also use the monotonicity of burning number under isometric subgraphs:
\begin{equation}\label{eq: monotonicity b}
\max\{b(C_{n,m-1}),\, b(C_{n-1,m})\}
\le b(C_{n,m})
\le
\min\{b(C_{n,m+1}),\, b(C_{n+1,m})\}.
\end{equation}
These comparisons allow us to propagate exact values once established in the balanced regime.

For our purposes, $\hat b$ provides a useful lower bound for the burning number of comb graphs. So we  next determine the structure of $\hat b_r(C_{n,m})$. By symmetry and the product structure of the comb, coverings may be assumed to use centers on the spine once $r$ is sufficiently large. If $r \ge m$, then any ball of radius $r$ centered at a spine vertex covers every vertex in each tooth within distance $r-m$ along the spine. Moving a center from a tooth to its adjacent spine vertex does not decrease coverage, so optimal coverings may be taken entirely on the spine. Thus,
\begin{equation}\label{eq:hatb_r_ge_m}
    \hat b_r(C_{n,m})
    =
    \left\lceil 
    \frac{n}{2(r-m+1)+1}
    \right\rceil,
    \qquad r \ge m.
\end{equation}
In particular, $\hat b_m(C_{n,m}) = n$.

Moreover, if $\floor{m/2}\le r < m$, a ball of radius $r$ centered on the spine covers at most one tooth completely. Hence $n$ balls (one per spine vertex) are necessary and sufficient, giving
\begin{equation}\label{eq:hatb_r_le_m}
    \hat b_r(C_{n,m}) = n,
    \qquad r = m-1,\ldots,\lfloor m/2\rfloor.
\end{equation}

For smaller radii, optimal coverings balance partial tooth coverage against residual spine coverage. Depending on how much of each tooth remains uncovered (after removing uniform segments of non-overlapping balls within each tooth), one either assigns additional centers within teeth or reduces to a smaller comb and iterates the argument. The resulting structure is summarized below.

\begin{lemma}\label{l: hat br}
For any positive integers $n,m$,
\[
\hat b_r(C_{n,m}) = A_r n + B_r,
\]
where $\displaystyle A_r=\left\lfloor \frac{m}{2r+1}\right\rfloor$, and writing $L=m-A_r(2r+1)$, we have
\[
B_r=
\begin{cases}
n, & L\ge r,\\[4pt]
\hat b_r(C_{n,L})
   =\displaystyle\left\lceil \frac{n}{2(r-L+1)+1}\right\rceil,
   & L<r.
\end{cases}
\]
\end{lemma}

This provides sufficient control to now establish the exact burning number for the balanced $n = m$ comb, with teeth and spine of the exact same order.

\begin{lemma}\label{l: n=m}
    $b(C_{n,n}) = n$.
\end{lemma}
\begin{proof}
    By \Cref{l: hat br}, then $\hat b_{n-1}(C_{n,n}) = n \le n$ and $\hat b_{n-2}(C_{n,n}) = n > n-1$
    so that $b(C_{n,n}) \ge \hat b(C_{n,n}) = n$. To establish the matching upper bound $b(C_{n,n}) \le n$,  it suffices to show that a successful burning sequence is generated by \textsc{GreedyComb}$(n; n,n; 1)$: placing the first fire in the top left corner of $C_{n,n}$ will burn the entire first tooth by time $n$ and each successive tooth moving right from the first tooth has one fewer vertex burned by time $n$ than the previous. So the remaining uncovered bottom segments of each tooth have $1,2,\ldots,n-1$ respective uncovered vertices. Recall the $i$-th fire ignited at time $i$ has burning radius $n - i$ by time $n$. So the next $n-1$ fires have respective burning radii of $n-2, n-3,\ldots,2,1,0$ by time $n$. It follows we can burn off the remaining segments by placing the each successive fire at the leaf node of the remaining outside tooth, moving inward one tooth at a time. This establishes then also $b(C_{n,n}) \le n$, so that $b(C_{n,n}) = n$.
\end{proof}

\begin{remark}
    From \Cref{l: n=m} and \Cref{eq: monotonicity b}, it follows  that necessarily 
    \begin{equation*}
        T_\gr \ge b(C_{n,m}) \ge \min\{b(C_{n,n}),b(C_{m,m})\} = \min\{n,m\}.
    \end{equation*}
\end{remark}

We next consider the burning number for when the teeth have one more node than the spine.

\begin{lemma}\label{l: n_n+1}
    $b(C_{n,n+1}) = n + 1$.
\end{lemma}
\begin{proof}
    Using \Cref{l: n=m} and \Cref{eq: monotonicity b}, we have $b(C_{n,n+1}) \le b(C_{n+1,n+1}) = n + 1$. Suppose we have a successful burning sequence on $C_{n,n+1}$ using only $n$ fires. Note the burning radii for each fire is $n-1,n-2,\ldots,1,0$. Since a leaf is distance $n+1$ away from the closest neighboring tooth, then any tooth without a fire will have an unburned leaf after $n$ rounds. So we must have at least one fire on each tooth and hence exactly one fire per tooth, which thus each burn a single leaf vertex. Suppose $v$ is the vertex above the leaf burned by the final fire that has burning radius 0. Then $v$ is not burned by the last fire while $v$ is distance $n$ to the closest tooth. So $v$ cannot be covered by any previous fire by time $n$. Hence, no such successful burning process using only $n$ fires exists. It follows $b(C_{n,n+1}) \ge n + 1$. 
\end{proof}

We now continue on, considering an additional vertex per tooth.

\begin{lemma}
    \label{l: n_n+2}
    $b(C_{n,n+2}) = n+1$.
\end{lemma}
\begin{proof}
    Using \Cref{l: n_n+1} and \Cref{eq: monotonicity b}, then $n + 1 = b(C_{n,n+1}) \le b(C_{n,n+2})$. We now show we can construct a successful burning sequence using $n + 1$ fires using \textsc{GreedyComb}$(n+1; n,n+2; 1)$ to then yield also $b(C_{n,n+2}) \le n+1$. The burning radii for the $n+1$ fires are $n,n-1,\ldots,1,0$. Placing the first fire at the top left corner burns precisely the top $n+1$ vertices on the first tooth, then $n$ of the next tooth, and so on, burning the top two vertices of the farthest tooth. The unburned vertices then constitute a path forest with $n$ segments consisting of $1,2,\ldots,n$ uncovered vertices, so placing the next $n$ fires at the leaf of the outside tooth, moving inward one fire at a time, will then burn the entire graph by time $n+1$. So $b(C_{n,n+2}) \le n+1$.
\end{proof}

We continue one more step, now considering $m = n+3$ vertices, to illustrate the added complexity of establishing the exact burning numbers for the tooth-dominant regime. Only the previous lemmas are used in the proof of \Cref{thm: bnc}.

\begin{proposition}\label{prop: n_n+3}
    $b(C_{n,n+3}) = n + 2$.
\end{proposition}
    
\begin{proof}
Using \Cref{l: n_n+1} and \Cref{eq: monotonicity b}, we have $b(C_{n,n+3}) \le b(C_{n+2,n+3}) = n+2$. To establish equality, we show that no successful burning sequence on $C_{n,n+3}$ can use only $n+1$ fires. 

Suppose for contradiction that such a sequence exists. Then the fires have radii $n,n-1,\ldots,1,0$. First, we note again necessarily every tooth must contain at least one fire source. Indeed, if some tooth contains no fire, then its leaf is at distance at least $n+1$ from any fire placed on another tooth, and hence remains unburned at time $n+1$. Since there are $n+1$ fires and $n$ teeth, exactly one tooth contains two fires (by the pigeonhole principle).

\begin{table}[t]
    \centering
    \begin{tabular}{|cc|ccc|}
        $n$ & $m$ & $\hat b(C_{n,m})$ & $b(C_{n,m})$ & $\lceil \sqrt{nm} \rceil$   \\ \hline
         $n$ & $n-5$ & $n-3$ & \cellcolor{gray!15}$n-3$ & \cellcolor{gray!15}$n-2$  \\
         $n$ & $n-4$ & $n-2$ & $n-2$ & $n-2$  \\
         $n$ & $n-3$ & $n-2$ & \cellcolor{gray!15}$n-2$ & \cellcolor{gray!15}$n-1$ \\
         $n$ & $n-2$ & $n-1$ & $n-1$ & $n-1$  \\
         $n$ & $n-1$ & $n$ & $n$ & $n$ \\
         $n$ & $n$ & $n$ & $n$ & $n$  \\
         $n$ & $n+1$ & $n$ & $n+1$ & $n+1$  \\
         $n$ & $n+2$ & $n$ & $n+1$ & $n+1$  \\
         $n$ & $n+3$ & $n$ & $n+2$ & $n+2$  \\
         $n$ & $n+4$ & $n$ & $n+2$ & $n+2$  \\
         $n$ & $n+5$ & $n$ & \cellcolor{gray!15}$n+2$ & \cellcolor{gray!15}$n+3$ 
    \end{tabular}
    \caption{Relation of $\hat b$, $b$ and the BNC bound for $|n-m| \le 5$ for $n \ge 10$.}
    \label{tab:small diff}
\end{table}

We now analyze possible placements of the first fire that has burning radius $n$.

\medskip
\noindent\textbf{Case 1: The first fire is placed on the spine.}  A fire of radius $n$ centered on the spine leaves on each tooth an uncovered segment consisting of at most the bottom three vertices. Thus, after the first fire, the uncovered graph consists of $n$ disjoint path segments, each of length at most 2. The remaining fires have radii $n-1,n-2,\ldots,0$, and hence their effective coverage along a path is at most $1,3,5,7,\ldots,2n-1$. Since the teeth are only connected by passing through the spine, each remaining segment must be cleared independently. However, the smaller-radius fires cannot eliminate all $n$ remaining segments of length 2 as one fire only has radius 0, so complete burning is impossible in this case.

\medskip
\noindent\textbf{Case 2: The first fire is placed one vertex below the spine.} In this case, one tooth is nearly burned, leaving a single uncovered vertex (the leaf), while each of the remaining teeth retains an uncovered segment of length at least $4$. Neighboring influence from other fires cannot reduce these segments below length $5$. Thus, we again obtain $n$ disjoint segments, which include one of length $1$ and all others of length at least $4$. Since the final two fires can only burn 3 and 1 vertices, respectively, the remaining radii are insufficient to eliminate all of these segments without assigning two fires to some tooth, which would leave another tooth uncovered. Hence burning is impossible also in this case.

\medskip
\noindent\textbf{Case 3: The first fire is placed two levels below the spine.} Here the entire tooth containing the first fire is burned (along with all spine vertices), while each of the remaining $n-1$ teeth retains an uncovered segment of length at least $5$. The remaining $n$ fires must burn these $n-1$ disjoint segments. The final two fires have effective path coverages of $1$ and $3$ vertices, respectively, and thus neither can burn a segment of length at least $5$ on its own. Moreover, they cannot be combined on a single tooth to eliminate such a segment. It follows that each must be paired with a larger fire in order to completely burn a tooth. This now requires at least two teeth to receive two fires each. However, with only $n$ remaining fires available to cover $n-1$ teeth, assigning two fires to two distinct teeth forces some tooth to receive none. That tooth therefore retains an unburned leaf after $n+1$ rounds. Hence complete burning is again impossible.

\medskip
Placing the first fire any lower only decreases its total coverage, while shifting the fire up reduces to one of the prior cases. This shows no placement of the first fire allows a successful burning sequence with $n+1$ fires. Therefore,
$b(C_{n,n+3}) \ge n+2$.
\end{proof}

We will stop at this point for working on exact burning numbers when $n \le m$. Using bounds for $T_\gr$ established later (see \Cref{prop: tooth-dominant}), a few additional exact values can be obtained (see \Cref{rmk: n+2}). Combining this also with exact results for the spine-dominant regime, we summarize the results in \Cref{tab:small diff} for $n \ge 10$ and $|n-m| \le 5$, where we also compare $b(C_{n,m})$ to $\hat b(C_{n,m})$ and $\lceil \sqrt{nm} \rceil$. 

We note that $n = 5$ marks a transition where $\lceil \sqrt{n(n+5)}\rceil = n + 3$ while for $n \le 4$ then $\lceil \sqrt{n(n+5)}\rceil = n + 2$. Hence, in the tooth-dominant regime, $m = n + 5$ is the first instance where the BNC upper bound ceases to be tight. In the spine-dominant regime, this first occurs at $m = n-3$. The bound is tight again at $m = n-4$, and fails once more at $m = n-5$ for $n \ge 10$ (see also \Cref{rmk: spine tight} and \Cref{fig:b_comparison}).

\subsection{\texorpdfstring{\boldmath$n \ge m$}{n ge m}: spine-dominant}
\label{sec: spine dom}

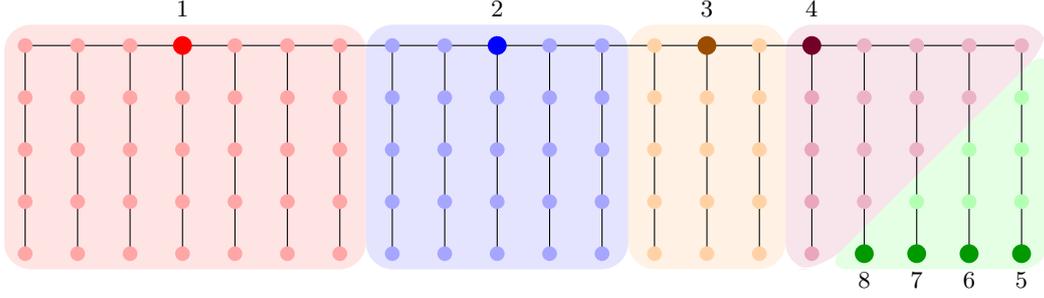
\begin{figure}[t]
\centering
\begin{tikzpicture}[scale=0.69]

\begin{scope}[xshift=0cm] 

\def\n{20}
\def\H{5}

\foreach \j in {1,...,\n} {
    \node[circle,fill=black,inner sep=1.5pt] (s\j) at (\j,-1) {};
}
\foreach \j in {1,...,19} {
    \draw (s\j) -- (s\the\numexpr\j+1\relax);
}

\foreach \j in {1,...,\n} {
    \foreach \i in {2,...,\H} {
        \node[circle,fill=black,inner sep=1.5pt] (v\i-\j) at (\j,-\i) {};
        \ifnum\i=2
            \draw (s\j) -- (v\i-\j);
        \else
            \draw (v\the\numexpr\i-1\relax-\j) -- (v\i-\j);
        \fi
    }
}

\begin{scope}[on background layer]

\fill[red!30,opacity=0.35,rounded corners=10pt] 
    (0.6,-.6) -- (7.5,-0.6) -- (7.5,-5.3) -- (0.6,-5.3) -- cycle;

\fill[blue!30,opacity=0.35,rounded corners=10pt]
    (7.5,-.6) -- (12.5,-.6) -- (12.5,-5.3) -- (7.5,-5.3) -- cycle;

\fill[orange!30,opacity=0.35,rounded corners=10pt]
    (12.5,-.6) -- (15.5,-.6) -- (15.5,-5.3) -- (12.5,-5.3) -- cycle;

\fill[purple!30,opacity=0.35,rounded corners=10pt]
    (15.5,-.6) -- (20.5,-.6) -- (20.5,-1) -- (16.2,-5.3) -- (15.5,-5.3) -- cycle;

\fill[green!30,opacity=0.35,rounded corners=10pt]
    (16.2,-5.3) -- (20.5,-1) -- (20.5,-5.3) -- cycle;

\end{scope}

\foreach \j in {1,...,7}{ \foreach \i in {1,...,5} { \fill[red!35] (\j,-\i) circle (4pt);} }
\node[circle,fill=red,inner sep=2.5pt] at (4,-1) {};

\foreach \j in {8,...,12}{ \foreach \i in {1,...,5} { \fill[blue!35] (\j,-\i) circle (4pt);} }
\node[circle,fill=blue,inner sep=2.5pt] at (10,-1) {};

\foreach \j in {13,...,15}{ \foreach \i in {1,...,5} { \fill[orange!35] (\j,-\i) circle (4pt);} }
\node[circle,fill=orange!60!black,inner sep=2.5pt] at (14,-1) {};

\foreach \j in {16,...,16}{ \foreach \i in {1,...,5} { \fill[purple!35] (\j,-\i) circle (4pt);} }
\node[circle,fill=purple!60!black,inner sep=2.5pt] at (16,-1) {};
\foreach \i in {1,...,4} { \fill[purple!30] (17,-\i) circle (4pt); }
\foreach \i in {1,...,3} { \fill[purple!30] (18,-\i) circle (4pt); }
\foreach \i in {1,...,2} { \fill[purple!30] (19,-\i) circle (4pt); }

\foreach \i in {1,...,1} { \fill[purple!30] (20,-\i) circle (4pt); }

\foreach \i in {2,...,5} { \fill[green!30] (20,-\i) circle (4pt); }
\foreach \i in {3,...,5} { \fill[green!30] (19,-\i) circle (4pt); }
\foreach \i in {4,...,5} { \fill[green!30] (18,-\i) circle (4pt); }

\foreach \i in {5,...,5} { \fill[green!30] (17,-\i) circle (4pt); }

\foreach \i in {17,...,20} {\node[circle,fill=green!60!black,inner sep=2.5pt] at (\i,-5) {};}


\node at (4,-0.3) {\small 1};
\node at (10,-0.3) {\small 2};
\node at (14,-0.3) {\small 3};
\node at (16,0-.3) {\small 4};
\node at (20,-5.5) {\small 5};
\node at (19,-5.5) {\small 6};
\node at (18,-5.5) {\small 7};
\node at (17,-5.5) {\small 8};

\end{scope}

\end{tikzpicture}
\caption{Greedy burning sequence for $C_{n,m}$ using $m-1 + \lceil \sqrt{n-m+1}\rceil = 8$ fires with $n = 20, m = 5$.}
\label{fig: 20_5}
\end{figure}

In this section on the spine-dominant regime where $n \ge m$, we establish the exact formula for the burning number
\[
b(C_{n,m}) = m - 1 + \left\lceil \sqrt{n-m+1}\right\rceil.
\]
This fully generalizes the known formula for paths, where $b(P_n) = \left\lceil \sqrt n\right\rceil$, which is now subsumed by the $m=1$ case (note $C_{n,1} = P_n$). 
\Cref{fig: 20_5} shows an optimal burning sequence on $C_{n,m}$ for $n = 20$ and $m = 5$, using $b(C_{n,m}) = 8$ fires (compare with the BNC upper bound $\lceil \sqrt{nm}\rceil = 10$). The numbered labels indicate the order of fire placements, and the vertex coloring identifies which fire first burns each vertex, with ties broken in favor of earlier fires. Note the final 4 fires burn disjoint segments of teeth, so we show the union of these disjoint regions that together comprise a triangular wedge in the bottom-right corner of $C_{n,m}$.

Our strategy will go as follows: establish $T_\gr$ takes the above form (in \Cref{prop: spine dom_greedy}). This already suffices to establish the BNC is true when $n \ge m$ (see \Cref{rmk: bnc spine-dominant}). We will then show any burning sequence on $C_{n,m}$ can be converted to a greedy burning sequence (in \Cref{prop: normalize_greedy_comb}) to thus establish this is an exact closed formula for the burning number in the spine-dominant regime.

\begin{proposition}\label{prop: spine dom_greedy}
For $n \ge m$, then
\[
T_\gr = m-1 + \left\lceil \sqrt{\,n - m + 1\,}\right\rceil.
\]
\end{proposition}

\begin{proof}
Let $T^* = m-1 + \left\lceil \sqrt{\,n - m + 1\,}\right\rceil$. We proceed by induction on $d=n-m\ge0$ to show $T_\gr = T^*$. First consider the base case
$d=0$, i.e., $n=m$. By \Cref{l: n=m}, we have $T_\gr=n$. Since $T^* =  n-1 + \lceil \sqrt{1} \rceil
= n,$ the statement holds.

Now assume the statement holds for all $C_{n',m}$ with $n'-m<d$, and consider
$C_{n,m}$ with $n-m=d\ge1$. We first show $T_\gr \le T^*$. Using \textsc{GreedyComb} (see \Cref{alg:greedy_comb}),
the first fire is placed on the spine at distance $T^*-1$ from the left-most
unburned leaf, and it burns completely exactly $2(T^*-m)+1$ teeth. Let
\[
n' = n - \bigl(2(T^*-m)+1\bigr)
\]
be the number of remaining teeth. To show $T_\gr \le T^*$, it suffices to burn $C_{n',m}$ using the remaining $T^*-1$ fires. Since we only need to ensure that all leaves are burned, any partial coverage on the upper vertices of the remaining teeth may be ignored: because $n-m \ge 1$, \textsc{GreedyComb} places $T^*-m+1 \ge 2$ spine fires, so at least one additional spine fire remains, and this fire covers all previously partially burned vertices. Hence the uncovered portion may be treated as a reduced comb $C_{n',m}$.

\smallskip
\textbf{Case 1:} $n'\ge m$.
Write $T^* = m-1 + a$, where
$a=\left\lceil \sqrt{n-m+1} \right\rceil$.
Then $2(T^*-m)+1 = 2(a-1)+1 = 2a-1$, so
$n' = n-(2a-1)$ and hence
$n'-m+1 = (n-m+1)-(2a-1)$.
Since $(a-1)^2 < n-m+1 \le a^2$, subtracting $2a-1$ gives
$(a-2)^2 < n'-m+1 \le (a-1)^2$,
so $\left\lceil \sqrt{n'-m+1} \right\rceil = a-1$.
Because $n'-m<d$, the inductive hypothesis applies and yields
\[
T_\gr({n',m})
=
m-1 + (a-1)
=
T^*-1.
\]
Thus the remaining comb is burned by the remaining fires.

\smallskip
\textbf{Case 2:} $n'<m$.
Then fewer than $m$ teeth remain. By construction, the unburned portions
consist of $n'$ vertical segments of lengths $1,2,\ldots,n'$.
The remaining $T^*-1$ fires have radii $T^*-2,T^*-3,\ldots$, and since
$T^* - m+1 \ge 2$, we have
$T^*-2 \ge m-1 \ge n'$.
Thus each of the next $n'$ fires can be placed at the leaf of a remaining
tooth and burns it entirely, and since $n'\le T^*-1$, all vertices are burned.

Combining both cases gives $T_\gr \le T^*$.

\medskip
To prove $T_\gr \ge T^*$, suppose only $T^*-1$ fires are allowed.
Write again $T^* = m-1+a$, where
$a=\left\lceil \sqrt{n-m+1} \right\rceil$.
Then the first fire removes
$2((T^*-1)-m)+1 = 2(a-2)+1 = 2a-3$
teeth, leaving
$n'' = n-(2a-3)$.
Now
\[
n''-m+1
=
(n-m+1)-(2a-3).
\]
Since $(a-1)^2 < n-m+1 \le a^2$, subtracting $2a-3$ gives $(a-2)^2 < n''-m+1$, and in particular $n''-m+1>0$; so $n''\ge m$ and the inductive hypothesis yields
\[
T_\gr({n'',m})
=
m-1+\left\lceil \sqrt{n''-m+1} \right\rceil.
\]
Moreover, since $n''-m+1>(a-2)^2$, we have
$\left\lceil \sqrt{n''-m+1} \right\rceil \ge a-1$, so also
\[
T_\gr({n'',m})
\ge m-1+(a-1)
= T^*-1.
\]
Thus, at least $T^*-1$ additional fires are required after the first fire,
but only $T^*-2$ remain available. This establishes $T^*-1$ fires do not suffice to burn $C_{n,m}$ using \textsc{GreedyComb},
so $T_\gr \ge T^*$.

We conclude that $T_\gr = T^*$. This completes the inductive step.
\end{proof}

\begin{remark}
    The argument above may be viewed as a variant of burning a path in which each
    vertex requires a fixed number of additional rounds before it is fully
    consumed. Rather than assigning vertices only two states (burned or
    unburned), one may equivalently initialize every vertex with a timer of
    $m$. When a vertex first catches fire, its timer decreases by one in each
    subsequent round until it reaches zero, at which point the vertex is fully
    burned. 
    
    Under this interpretation, the greedy strategy reduces to a timer-delayed burning
    process on the path formed by the leaves, with the first $m$ rounds accounting
    for the delay induced by the tooth lengths. The final $m-1$ fires then serve
    only to decrease the remaining timers and cannot propagate further along the
    path.
\end{remark}

\begin{remark}[BNC: spine-dominant]\label{rmk: bnc spine-dominant}
By \Cref{prop: spine dom_greedy}, we have $b(C_{n,m}) \le T_{\gr}$. This already suffices to prove the BNC is true in the spine-dominant regime: it suffices to prove the inequality without ceiling functions
\[
m-1 + \sqrt{n-m+1} \le \sqrt{nm}.
\]
This inequality always holds for $n \ge m$ as using only simple algebraic manipulations reduces this to being equivalent to showing $2 \sqrt{nm} \le n + m$, which always holds by AM--GM.
\end{remark}

Our next goal is now to show any burning sequence can be converted to a greedy burning sequence using an equal or fewer number of rounds. \Cref{fig:C74_normalization} shows an instance of a burning sequence on $C_{n,m}$ for $n = 7,m=4$ using $T = 5$ fires that is transformed to the burning sequence produced by \textsc{GreedyComb}$(5; 7,4;1)$. The following proposition outlines a procedure to carry forward this conversion of any burning sequence to a greedy burning sequence.

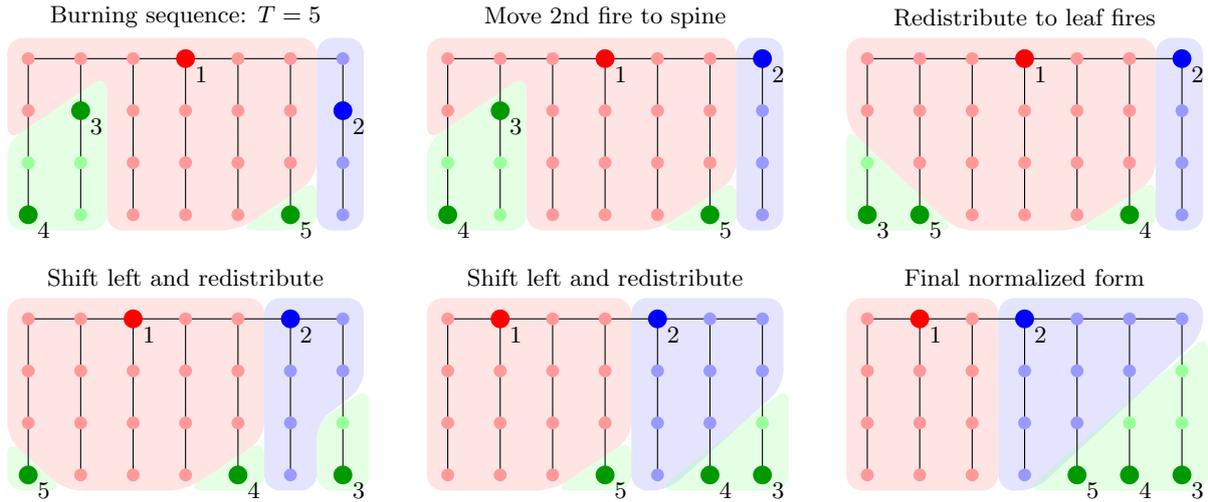
\begin{figure}[t]
\centering
\begin{tikzpicture}[scale=0.69]

\def\n{7}
\def\m{4}


\begin{scope}[xshift=0cm]

\def\n{7}
\def\m{4}


\foreach \j in {1,...,\n} {
    \node[circle,fill=black,inner sep=1.4pt] (s\j) at (\j,-1) {};
}
\foreach \j in {1,...,6} {
    \draw (s\j) -- (s\the\numexpr\j+1\relax);
}

\foreach \j in {1,...,\n} {
    \foreach \i in {2,...,\m} {
        \node[circle,fill=black,inner sep=1.4pt] (v\i-\j) at (\j,-\i) {};
        \ifnum\i=2
            \draw (s\j) -- (v\i-\j);
        \else
            \draw (v\the\numexpr\i-1\relax-\j) -- (v\i-\j);
        \fi
    }
}


\begin{scope}[on background layer]

\fill[red!30,opacity=0.35,rounded corners=6pt]
    (.6,-0.6) -- (6.5,-0.6) -- (6.5,-3.3) -- (5,-4.3) -- (2.5,-4.3) -- (2.5,-1.3) -- (.6,-2.6) -- cycle;
\fill[blue!30,opacity=0.35,rounded corners=6pt]
    (6.5,-0.6) -- (7.4,-0.6) -- (7.4,-4.3) -- (6.5,-4.3) -- cycle;
\fill[green!30,opacity=0.35,rounded corners=6pt]
    (0.6,-2.6) -- (2.5,-1.3) -- (2.5,-4.3) -- (0.6,-4.3) -- cycle;
\fill[green!30,opacity=0.35,rounded corners=6pt]
    (5,-4.3) -- (6.5,-3.3) -- (6.5,-4.3) -- cycle;

\end{scope}


\foreach \i in {1,...,2} { \fill[red!40] (1,-\i) circle (3.5pt); }
\foreach \i in {1,...,1} { \fill[red!40] (2,-\i) circle (3.5pt); }
\foreach \i in {1,...,4} { \fill[red!40] (3,-\i) circle (3.5pt); }
\foreach \i in {1,...,4} { \fill[red!40] (4,-\i) circle (3.5pt); }
\foreach \i in {1,...,4} { \fill[red!40] (5,-\i) circle (3.5pt); }
\foreach \i in {1,...,3} { \fill[red!40] (6,-\i) circle (3.5pt); }

\foreach \i in {1,...,4} { \fill[blue!40] (7,-\i) circle (3.5pt); }

\foreach \i in {3,...,4} { \fill[green!40] (1,-\i) circle (3.5pt); }
\foreach \i in {2,...,4} { \fill[green!40] (2,-\i) circle (3.5pt); }
\foreach \i in {4,...,4} { \fill[green!40] (6,-\i) circle (3.5pt); }


\node[circle,fill=red,inner sep=2.5pt] at (4,-1) {};
\node[circle,fill=blue,inner sep=2.5pt] at (7,-2) {};
\node[circle,fill=green!60!black,,inner sep=2.5pt] at (1,-4) {};
\node[circle,fill=green!60!black,,inner sep=2.5pt] at (2,-2) {};
\node[circle,fill=green!60!black,,inner sep=2.5pt] at (6,-4) {};


\node at (4.3,-1.3) {\small 1};
\node at (7.3,-2.3) {\small 2};
\node at (2.3,-2.3) {\small 3};
\node at (1.3,-4.3) {\small 4};
\node at (6.3,-4.3) {\small 5};


\node at (4,-0.2) {\small Burning sequence: $T=5$};

\end{scope}

\begin{scope}[xshift=8cm]

\def\n{7}
\def\m{4}


\foreach \j in {1,...,\n} {
    \node[circle,fill=black,inner sep=1.4pt] (s\j) at (\j,-1) {};
}
\foreach \j in {1,...,6} {
    \draw (s\j) -- (s\the\numexpr\j+1\relax);
}

\foreach \j in {1,...,\n} {
    \foreach \i in {2,...,\m} {
        \node[circle,fill=black,inner sep=1.4pt] (v\i-\j) at (\j,-\i) {};
        \ifnum\i=2
            \draw (s\j) -- (v\i-\j);
        \else
            \draw (v\the\numexpr\i-1\relax-\j) -- (v\i-\j);
        \fi
    }
}


\begin{scope}[on background layer]

\fill[red!30,opacity=0.35,rounded corners=6pt]
    (.6,-0.6) -- (6.5,-0.6) -- (6.5,-3.3) -- (5,-4.3) -- (2.5,-4.3) -- (2.5,-1.3) -- (.6,-2.6) -- cycle;
\fill[blue!30,opacity=0.35,rounded corners=6pt]
    (6.5,-0.6) -- (7.4,-0.6) -- (7.4,-4.3) -- (6.5,-4.3) -- cycle;
\fill[green!30,opacity=0.35,rounded corners=6pt]
    (0.6,-2.6) -- (2.5,-1.3) -- (2.5,-4.3) -- (0.6,-4.3) -- cycle;
\fill[green!30,opacity=0.35,rounded corners=6pt]
    (5,-4.3) -- (6.5,-3.3) -- (6.5,-4.3) -- cycle;

\end{scope}


\foreach \i in {1,...,2} { \fill[red!40] (1,-\i) circle (3.5pt); }
\foreach \i in {1,...,1} { \fill[red!40] (2,-\i) circle (3.5pt); }
\foreach \i in {1,...,4} { \fill[red!40] (3,-\i) circle (3.5pt); }
\foreach \i in {1,...,4} { \fill[red!40] (4,-\i) circle (3.5pt); }
\foreach \i in {1,...,4} { \fill[red!40] (5,-\i) circle (3.5pt); }
\foreach \i in {1,...,3} { \fill[red!40] (6,-\i) circle (3.5pt); }

\foreach \i in {1,...,4} { \fill[blue!40] (7,-\i) circle (3.5pt); }

\foreach \i in {3,...,4} { \fill[green!40] (1,-\i) circle (3.5pt); }
\foreach \i in {2,...,4} { \fill[green!40] (2,-\i) circle (3.5pt); }
\foreach \i in {4,...,4} { \fill[green!40] (6,-\i) circle (3.5pt); }


\node[circle,fill=red,inner sep=2.5pt] at (4,-1) {};
\node[circle,fill=blue,inner sep=2.5pt] at (7,-1) {};
\node[circle,fill=green!60!black,,inner sep=2.5pt] at (1,-4) {};
\node[circle,fill=green!60!black,,inner sep=2.5pt] at (2,-2) {};
\node[circle,fill=green!60!black,,inner sep=2.5pt] at (6,-4) {};


\node at (4.3,-1.3) {\small 1};
\node at (7.3,-1.3) {\small 2};
\node at (2.3,-2.3) {\small 3};
\node at (1.3,-4.3) {\small 4};
\node at (6.3,-4.3) {\small 5};


\node at (4,-0.2) {\small Move 2nd fire to spine};

\end{scope}

\begin{scope}[xshift=16cm]

\def\n{7}
\def\m{4}


\foreach \j in {1,...,\n} {
    \node[circle,fill=black,inner sep=1.4pt] (s\j) at (\j,-1) {};
}
\foreach \j in {1,...,6} {
    \draw (s\j) -- (s\the\numexpr\j+1\relax);
}

\foreach \j in {1,...,\n} {
    \foreach \i in {2,...,\m} {
        \node[circle,fill=black,inner sep=1.4pt] (v\i-\j) at (\j,-\i) {};
        \ifnum\i=2
            \draw (s\j) -- (v\i-\j);
        \else
            \draw (v\the\numexpr\i-1\relax-\j) -- (v\i-\j);
        \fi
    }
}


\begin{scope}[on background layer]

\fill[red!30,opacity=0.35,rounded corners=6pt]
    (.6,-0.6) -- (6.5,-0.6) -- (6.5,-3.3) -- (5,-4.3) -- (2.7,-4.3) -- (.6,-2.3) -- cycle;
\fill[blue!30,opacity=0.35,rounded corners=6pt]
    (6.5,-0.6) -- (7.4,-0.6) -- (7.4,-4.3) -- (6.5,-4.3) -- cycle;
\fill[green!30,opacity=0.35,rounded corners=6pt]
    (0.6,-2.3) -- (2.7,-4.3) -- (0.6,-4.3) -- cycle;
\fill[green!30,opacity=0.35,rounded corners=6pt]
    (5,-4.3) -- (6.5,-3.3) -- (6.5,-4.3) -- cycle;

\end{scope}


\foreach \i in {1,...,2} { \fill[red!40] (1,-\i) circle (3.5pt); }
\foreach \i in {1,...,3} { \fill[red!40] (2,-\i) circle (3.5pt); }
\foreach \i in {1,...,4} { \fill[red!40] (3,-\i) circle (3.5pt); }
\foreach \i in {1,...,4} { \fill[red!40] (4,-\i) circle (3.5pt); }
\foreach \i in {1,...,4} { \fill[red!40] (5,-\i) circle (3.5pt); }
\foreach \i in {1,...,3} { \fill[red!40] (6,-\i) circle (3.5pt); }

\foreach \i in {1,...,4} { \fill[blue!40] (7,-\i) circle (3.5pt); }

\foreach \i in {3,...,4} { \fill[green!40] (1,-\i) circle (3.5pt); }
\foreach \i in {4,...,4} { \fill[green!40] (2,-\i) circle (3.5pt); }
\foreach \i in {4,...,4} { \fill[green!40] (6,-\i) circle (3.5pt); }


\node[circle,fill=red,inner sep=2.5pt] at (4,-1) {};
\node[circle,fill=blue,inner sep=2.5pt] at (7,-1) {};
\node[circle,fill=green!60!black,,inner sep=2.5pt] at (1,-4) {};
\node[circle,fill=green!60!black,,inner sep=2.5pt] at (2,-4) {};
\node[circle,fill=green!60!black,,inner sep=2.5pt] at (6,-4) {};




\node at (4.3,-1.3) {\small 1};
\node at (7.3,-1.3) {\small 2};
\node at (1.3,-4.3) {\small 3};
\node at (6.3,-4.3) {\small 4};
\node at (2.3,-4.3) {\small 5};


\node at (4,-0.2) {\small Redistribute to leaf fires};

\end{scope}


\begin{scope}[yshift=-5cm]

\def\n{7}
\def\m{4}


\foreach \j in {1,...,\n} {
    \node[circle,fill=black,inner sep=1.4pt] (s\j) at (\j,-1) {};
}
\foreach \j in {1,...,6} {
    \draw (s\j) -- (s\the\numexpr\j+1\relax);
}

\foreach \j in {1,...,\n} {
    \foreach \i in {2,...,\m} {
        \node[circle,fill=black,inner sep=1.4pt] (v\i-\j) at (\j,-\i) {};
        \ifnum\i=2
            \draw (s\j) -- (v\i-\j);
        \else
            \draw (v\the\numexpr\i-1\relax-\j) -- (v\i-\j);
        \fi
    }
}


\begin{scope}[on background layer]

\fill[red!30,opacity=0.35,rounded corners=6pt]
    (.6,-0.6) -- (5.5,-0.6) -- (5.5,-3.3) -- (4,-4.3) -- (1.7,-4.3) -- (.6,-3.3) -- cycle;
\fill[blue!30,opacity=0.35,rounded corners=6pt]
    (5.5,-0.6) -- (7.4,-0.6) -- (7.4,-2.3) -- (6.5,-3) -- (6.5,-4.3) -- (5.5,-4.3)-- cycle;
\fill[green!30,opacity=0.35,rounded corners=6pt]
    (0.6,-3.3) -- (1.7,-4.3) -- (0.6,-4.3) -- cycle;
\fill[green!30,opacity=0.35,rounded corners=6pt]
    (4,-4.3) -- (5.5,-3.3) -- (5.5,-4.3) -- cycle;
\fill[green!30,opacity=0.35,rounded corners=6pt]
    (6.5,-4.3) -- (6.5,-3) -- (7.5,-2.3) -- (7.5,-4.3) -- cycle;

\end{scope}


\foreach \i in {1,...,3} { \fill[red!40] (1,-\i) circle (3.5pt); }
\foreach \i in {1,...,4} { \fill[red!40] (2,-\i) circle (3.5pt); }
\foreach \i in {1,...,4} { \fill[red!40] (3,-\i) circle (3.5pt); }
\foreach \i in {1,...,4} { \fill[red!40] (4,-\i) circle (3.5pt); }
\foreach \i in {1,...,3} { \fill[red!40] (5,-\i) circle (3.5pt); }

\foreach \i in {1,...,4} { \fill[blue!40] (6,-\i) circle (3.5pt); }
\foreach \i in {1,...,2} { \fill[blue!40] (7,-\i) circle (3.5pt); }

\foreach \i in {4,...,4} { \fill[green!40] (1,-\i) circle (3.5pt); }
\foreach \i in {4,...,4} { \fill[green!40] (5,-\i) circle (3.5pt); }
\foreach \i in {3,...,4} { \fill[green!40] (7,-\i) circle (3.5pt); }


\node[circle,fill=red,inner sep=2.5pt] at (3,-1) {};
\node[circle,fill=blue,inner sep=2.5pt] at (6,-1) {};
\node[circle,fill=green!60!black,,inner sep=2.5pt] at (1,-4) {};
\node[circle,fill=green!60!black,,inner sep=2.5pt] at (5,-4) {};
\node[circle,fill=green!60!black,,inner sep=2.5pt] at (7,-4) {};




\node at (3.3,-1.3) {\small 1};
\node at (6.3,-1.3) {\small 2};
\node at (7.3,-4.3) {\small 3};
\node at (5.3,-4.3) {\small 4};
\node at (1.3,-4.3) {\small 5};


\node at (4,-0.2) {\small Shift left and redistribute};

\end{scope}

\begin{scope}[xshift=8cm,yshift=-5cm]

\def\n{7}
\def\m{4}


\foreach \j in {1,...,\n} {
    \node[circle,fill=black,inner sep=1.4pt] (s\j) at (\j,-1) {};
}
\foreach \j in {1,...,6} {
    \draw (s\j) -- (s\the\numexpr\j+1\relax);
}

\foreach \j in {1,...,\n} {
    \foreach \i in {2,...,\m} {
        \node[circle,fill=black,inner sep=1.4pt] (v\i-\j) at (\j,-\i) {};
        \ifnum\i=2
            \draw (s\j) -- (v\i-\j);
        \else
            \draw (v\the\numexpr\i-1\relax-\j) -- (v\i-\j);
        \fi
    }
}


\begin{scope}[on background layer]

\fill[red!30,opacity=0.35,rounded corners=6pt]
    (.6,-0.6) -- (4.5,-0.6) -- (4.5,-3.3) -- (3,-4.3) -- (.6,-4.3) -- cycle;
\fill[blue!30,opacity=0.35,rounded corners=6pt]
    (4.5,-0.6) -- (7.4,-0.6) -- (7.4,-2.3) -- (5.2,-4.3) -- (4.5,-4.3)-- cycle;
\fill[green!30,opacity=0.35,rounded corners=6pt]
    (3,-4.3) -- (4.5,-3.3) -- (4.5,-4.3) -- cycle;
\fill[green!30,opacity=0.35,rounded corners=6pt]
    (5,-4.3) -- (7.5,-2.3) -- (7.5,-4.3) -- cycle;

\end{scope}


\foreach \i in {1,...,4} { \fill[red!40] (1,-\i) circle (3.5pt); }
\foreach \i in {1,...,4} { \fill[red!40] (2,-\i) circle (3.5pt); }
\foreach \i in {1,...,4} { \fill[red!40] (3,-\i) circle (3.5pt); }
\foreach \i in {1,...,3} { \fill[red!40] (4,-\i) circle (3.5pt); }

\foreach \i in {1,...,4} { \fill[blue!40] (5,-\i) circle (3.5pt); }
\foreach \i in {1,...,3} { \fill[blue!40] (6,-\i) circle (3.5pt); }
\foreach \i in {1,...,2} { \fill[blue!40] (7,-\i) circle (3.5pt); }

\foreach \i in {4,...,4} { \fill[green!40] (4,-\i) circle (3.5pt); }
\foreach \i in {4,...,4} { \fill[green!40] (6,-\i) circle (3.5pt); }
\foreach \i in {3,...,4} { \fill[green!40] (7,-\i) circle (3.5pt); }


\node[circle,fill=red,inner sep=2.5pt] at (2,-1) {};
\node[circle,fill=blue,inner sep=2.5pt] at (5,-1) {};
\node[circle,fill=green!60!black,,inner sep=2.5pt] at (4,-4) {};
\node[circle,fill=green!60!black,,inner sep=2.5pt] at (6,-4) {};
\node[circle,fill=green!60!black,,inner sep=2.5pt] at (7,-4) {};




\node at (2.3,-1.3) {\small 1};
\node at (5.3,-1.3) {\small 2};
\node at (7.3,-4.3) {\small 3};
\node at (6.3,-4.3) {\small 4};
\node at (4.3,-4.3) {\small 5};


\node at (4,-0.2) {\small Shift left and redistribute};

\end{scope}

\begin{scope}[xshift=16cm,yshift=-5cm]

\def\n{7}
\def\m{4}


\foreach \j in {1,...,\n} {
    \node[circle,fill=black,inner sep=1.4pt] (s\j) at (\j,-1) {};
}
\foreach \j in {1,...,6} {
    \draw (s\j) -- (s\the\numexpr\j+1\relax);
}

\foreach \j in {1,...,\n} {
    \foreach \i in {2,...,\m} {
        \node[circle,fill=black,inner sep=1.4pt] (v\i-\j) at (\j,-\i) {};
        \ifnum\i=2
            \draw (s\j) -- (v\i-\j);
        \else
            \draw (v\the\numexpr\i-1\relax-\j) -- (v\i-\j);
        \fi
    }
}


\begin{scope}[on background layer]

\fill[red!30,opacity=0.35,rounded corners=6pt]
    (.6,-0.6) -- (3.5,-0.6) -- (3.5,-4.3) -- (.6,-4.3) -- cycle;
\fill[blue!30,opacity=0.35,rounded corners=6pt]
    (3.5,-0.6) -- (7.4,-0.6) -- (7.4,-1.3) -- (4.2,-4.3) -- (3.5,-4.3)-- cycle;
\fill[green!30,opacity=0.35,rounded corners=6pt]
    (4,-4.3) -- (7.5,-1.3) -- (7.5,-4.3) -- cycle;

\end{scope}


\foreach \i in {1,...,4} { \fill[red!40] (1,-\i) circle (3.5pt); }
\foreach \i in {1,...,4} { \fill[red!40] (2,-\i) circle (3.5pt); }
\foreach \i in {1,...,4} { \fill[red!40] (3,-\i) circle (3.5pt); }

\foreach \i in {1,...,4} { \fill[blue!40] (4,-\i) circle (3.5pt); }
\foreach \i in {1,...,3} { \fill[blue!40] (5,-\i) circle (3.5pt); }
\foreach \i in {1,...,2} { \fill[blue!40] (6,-\i) circle (3.5pt); }
\foreach \i in {1,...,1} { \fill[blue!40] (7,-\i) circle (3.5pt); }

\foreach \i in {4,...,4} { \fill[green!40] (5,-\i) circle (3.5pt); }
\foreach \i in {3,...,4} { \fill[green!40] (6,-\i) circle (3.5pt); }
\foreach \i in {2,...,4} { \fill[green!40] (7,-\i) circle (3.5pt); }


\node[circle,fill=red,inner sep=2.5pt] at (2,-1) {};
\node[circle,fill=blue,inner sep=2.5pt] at (4,-1) {};
\node[circle,fill=green!60!black,,inner sep=2.5pt] at (5,-4) {};
\node[circle,fill=green!60!black,,inner sep=2.5pt] at (6,-4) {};
\node[circle,fill=green!60!black,,inner sep=2.5pt] at (7,-4) {};




\node at (2.3,-1.3) {\small 1};
\node at (4.3,-1.3) {\small 2};
\node at (7.3,-4.3) {\small 3};
\node at (6.3,-4.3) {\small 4};
\node at (5.3,-4.3) {\small 5};


\node at (4,-0.2) {\small Final normalized form};

\end{scope}

\end{tikzpicture}
\caption{Normalization example for $C_{7,4}$, starting with a burning sequence with $T = 5$ fires, that is transformed to $\textsc{GreedyComb}(5;7,4;1)$ by first moving early fires to spine (Step 1), reassignment of late fires to leaf locations (Step 5), and then three shift-left by 1 all fires to the right of a leaf fire and redistribute leaf fires steps (Step 6).}
\label{fig:C74_normalization}
\end{figure}

\begin{proposition}[Normalization to \textsc{GreedyComb}]\label{prop: normalize_greedy_comb}
Let $n \ge m$ and suppose $\mathcal B$ is a successful burning sequence
that burns $C_{n,m}$ in $T$ rounds.
Then $\mathcal B$ can be transformed, without loss of coverage,
into the burning sequence generated by $\textsc{GreedyComb}(T;n,m;1)$.
\end{proposition}

\begin{proof}
Label vertices as $(i,j)$,
where $j\in\{1,\dots,n\}$ indexes the tooth
and $i\in\{1,\dots,m\}$ indexes height along the tooth,
with $(1,j)$ the spine vertex
and $(m,j)$ the leaf. Label the burning sequence $\mathcal B=\{(v_t,r_t)\}_{t=1}^T$
with burning radii $r_t=T-t$.

Our goal is to describe an explicit sequence of steps that transforms this initial burning sequence into the one produced by \textsc{GreedyComb}$(T;n,m;1)$. By construction, that sequence is determined by placing the first $T-m+1$ fires along the spine, maximally spaced to burn the left-most teeth from left to right, and then placing the final $m-1$ fires at leaf vertices in the lower right of the comb, proceeding from right to left to eliminate any remaining uncovered segments.

\medskip
\noindent
\emph{Step 0: Reduction.}
If some fire $(v_s,r_s)$ burns no vertex
not already burned by earlier fires,
remove it and increase the radii of all later fires by $1$.
This corresponds to placing all later fires one round earlier (so burning still completes after $T$ rounds). This will be repeated after each following step.

\medskip
\noindent
\emph{Step 1: Move large fires to the spine.}
If $r_t \ge m-1$,
move the fire to the spine vertex $(1,j)$ of its tooth.
Since the leaf $(m,j)$ is distance $m-1$ from the spine,
this preserves vertical reach and weakly increases horizontal coverage.
After each such move, apply Step~0 if needed.
Hence all fires with radius at least $m-1$ lie on spine vertices; these are precisely the first $T-m+1$ fires.

\medskip
\noindent
\emph{Step 2: Boundary tightening.}
A spine fire at tooth $j$ completely burns every tooth
within horizontal distance $r_t-(m-1)$ of $j$.
If such a fire extends beyond tooth $1$ or $n$,
shift it inward so that one extremal burned tooth
is exactly $1$ or $n$.
If this is impossible,
then that fire burns all teeth of the comb.
In this case, all later fires are redundant by Step~0,
and the configuration reduces to a single spine fire
whose placement may be chosen to coincide
with the first placement of \textsc{GreedyComb}.

\medskip
\noindent
\emph{Step 3: Spreading spine fires.}
Suppose two consecutive spine fires both completely burn the same tooth. 
Shift the left fire as far left as possible and, if necessary, shift the right fire as far right as possible so that no tooth is fully burned by both fires. If such a separation is impossible, then their union already burns the entire comb. In this case, move the left fire so that its leftmost fully burned tooth is the first tooth, and then shift the right fire as far right as possible. 

After finitely many such adjustments, the collections of fully burned teeth
corresponding to distinct spine fires are pairwise disjoint (or the entire comb is consumed by spine fires).
We may then reorder the spine fires so that earlier fires lie to the left of later fires.
Each spine fire induces a trapezoidal burning region: a contiguous block of fully burned teeth together with triangular boundary portions on adjacent teeth.
When two such regions abut along their fully burned blocks,
their union is again a trapezoid.
So this reordering preserves the total union of burned vertices and does not change overall coverage.

The spine fires therefore form one or more disjoint blocks
separated by triangular uncovered regions along the leaves of teeth not fully consumed by a spine fire. If the spine fires  already consume the entire graph, then go to Step 6.

\medskip
\noindent
\emph{Step 4: Counting uncovered leaves.}
Let $n'$ be the number of remaining fires (after reduction).
Since each remaining fire can burn at most one leaf,
and the input burning sequence succeeds,
the number of uncovered leaves after the spine fires
is at most $n' \le m-1$.
Moreover, these remaining fires have radii at most
$m-2,m-3,\dots,1,0$.

\medskip
\noindent
\emph{Step 5: Greedy reassignment on uncovered segments.}
Each uncovered region forms a vertical segment
in a tooth adjacent to a boundary of spine coverage.
Assign the remaining $n'$ fires greedily:
place the earliest remaining fire
on the leaf of the largest uncovered segment,
breaking ties by choosing the right-most such segment.
Proceed inductively.

Since $n'$ bounds the number of uncovered leaves,
this assignment burns all remaining segments,
using exactly one fire per segment.
Thus each tooth contains at most one fire,
and every fire is either a spine fire or a leaf fire.

\medskip
\noindent
\emph{Step 6: Left-shift normalization.}
If some leaf fire lies strictly to the left of another,
let $\ell$ be the leftmost such leaf fire.
Shift every fire strictly to the right of $\ell$
one tooth to the left.

Because the spine fires collectively burn
a fixed number of fully burned teeth,
this shift preserves that number.
The rightmost boundary of the spine coverage
has triangular slope $-1$,
so shifting left by one tooth removes the leftmost uncovered leaf
and creates exactly one new uncovered leaf at the right boundary.
Hence the number of leaves unburned by spine fires is preserved.

Reassign the leaf fires greedily as in Step~5.
Coverage by time $T$ is preserved,
since fully burned teeth are unchanged
and only boundary leaves are reallocated. 

Each shift strictly decreases
the number of leaf fires lying left of the rightmost block.
Since there are at most $m-1$ leaf fires,
this process terminates after finitely many steps.

\medskip
At termination,
all leaves not burned by spine fires
lie on the rightmost teeth.
The spine fires occupy consecutive spine vertices,
ordered by decreasing radius from left to right.
The remaining leaf fires occupy the rightmost teeth,
ordered by decreasing radius from right to left.

This configuration is precisely the burning sequence produced
\textsc{GreedyComb}$(T;n,m;1)$.
\end{proof}

In particular, we now note that an optimal burning sequence on $C_{n,m}$ can be converted to a greedy burning sequence using the same number of rounds. We are now able to complete a proof of \Cref{prop: spine-dominant} that yields the closed form of $b(C_{n,m})$ for the spine-dominant regime:

\begin{proof}[Proof of \Cref{prop: spine-dominant}]
    Let $T^* = m-1 + \lceil \sqrt{n-m + 1}\rceil$. For $n \ge m$, we have $b(C_{n,m}) \le T^*$ by \Cref{prop: spine dom_greedy} and $b(C_{n,m}) \ge T^*$ by \Cref{prop: normalize_greedy_comb}. Hence, $b(C_{n,m}) = T^*$.
\end{proof}

\begin{remark}
    [When is the BNC bound sharp?]\label{rmk: spine tight}

\begin{figure}[t]
    \centering
    \includegraphics[width=0.4\linewidth]{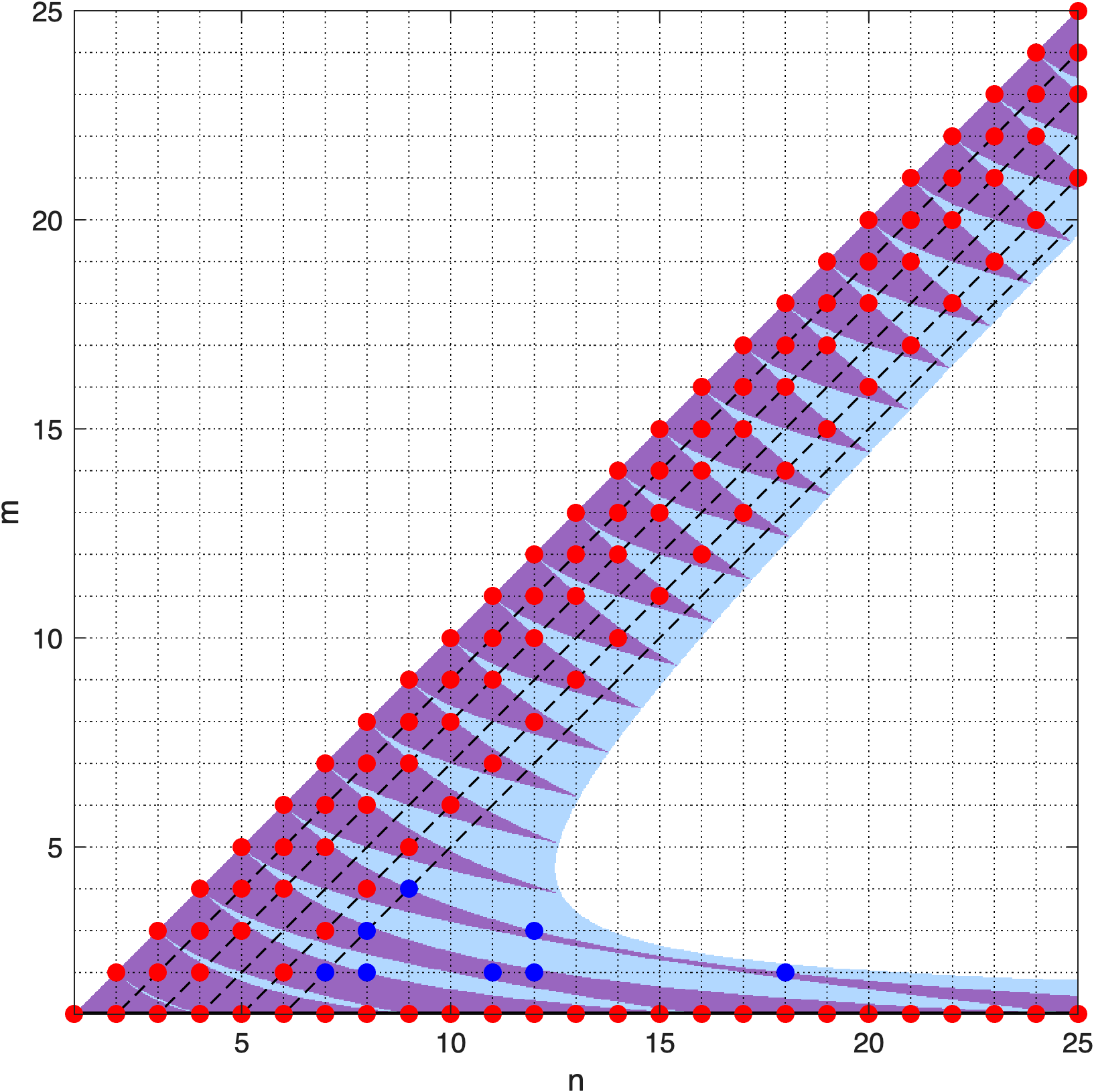}
    \caption{Tightness of the BNC bound in the spine-dominant regime. 
    The shaded blue region denotes the feasibility region $\mathcal F$. 
    The serrated curves indicate the real solutions of equality under the ceiling, 
    and the integer lattice points satisfying tightness are highlighted: 
    infinite families in red and isolated exceptional points in blue.}
    
    \label{fig:sharp BNC n larger}
\end{figure}

    We now have a closed form for the burning number for the spine-dominant regime. We know the BNC bound is tight at the boundary cases $m = 1$ and $n = m$, so we outline all remaining cases when this bound is tight. (See also \Cref{fig:b_comparison}.)

Recall that for all $n \ge m$ we have
\[
\sqrt{nm} \ge m - 1 + \sqrt{n-m+1},
\]
and the burning number and BNC bound are obtained by applying the ceiling function. Equality under the ceiling occurs only when the difference is less than 1 and both values lie in the same integer interval. 

So as an initial check, we define a {feasibility region}
\[
\mathcal F = \{(x,y) : x \ge y \ge 1, f(x,y) < 1\},
\]
where $f(x,y) = \sqrt{xy} - (y-1 + \sqrt{x-y+1})$. Thus, any $(n,m)$ satisfying the BNC bound is tight must lie in $\mathcal F$. Using simple bounds on $P$ and the concavity of $f(x,y)$ in $x$, one sees that only finitely many integer points outside the trivial cases need to be checked for $n = m + c$ and $5 \le c \le 20$. A straightforward computation (omitted for brevity) identifies the additional exceptions:
\[
(n,m) \in \{ (7,2),(8,3),(9,4),(8,2),(11,2),(12,3),(12,2),(18,2)\}.
\]
For the remaining infinite families, it suffices to consider $n = m + c$ for $1 \le c \le 4$. Checking these yields equality exactly when $c \in \{1,2,4\}$ and $m \ge 1$, while $c = 3$ gives equality only for $m = 1$. This completely characterizes when the BNC bound is sharp in the $n \ge m$ regime, as illustrated in \Cref{fig:sharp BNC n larger}.

\end{remark}

\subsection{\texorpdfstring{\boldmath$n \le m$}{n le m}: tooth-dominant}
\label{sec: tooth dom}

\begin{figure}[t]
\centering
\begin{tikzpicture}[scale=0.69]

\begin{scope}[xshift=0cm] 

\def\n{4}
\def\H{15}

\foreach \j in {1,...,\n} {
    \node[circle,fill=black,inner sep=1.5pt] (s\j) at (\j,-1) {};
}
\foreach \j in {1,...,3} {
    \draw (s\j) -- (s\the\numexpr\j+1\relax);
}

\foreach \j in {1,...,\n} {
    \foreach \i in {2,...,\H} {
        \node[circle,fill=black,inner sep=1.5pt] (v\i-\j) at (\j,-\i) {};
        \ifnum\i=2
            \draw (s\j) -- (v\i-\j);
        \else
            \draw (v\the\numexpr\i-1\relax-\j) -- (v\i-\j);
        \fi
        \ifnum\i=\H
            \draw (v\i-\j) -- ++(0,-0.5);
        \fi
    }
}

\begin{scope}[on background layer]

\fill[red!30,opacity=0.35,rounded corners=10pt]
    (0.6,-.6) -- (4.4,-0.6) -- (4.4,-2.2) -- (0.6,-5.7) -- cycle;

\fill[blue!30,opacity=0.35,rounded corners=10pt]
    (0.6,-5.8) -- (4.4,-2.3) -- (4.4,-9.8) -- (0.6,-6.2) -- cycle;

\end{scope}

\foreach \i in {1,...,5} { \fill[red!35] (1,-\i) circle (4pt); }
\foreach \i in {1,...,4} { \fill[red!35] (2,-\i) circle (4pt); }
\foreach \i in {1,...,3} { \fill[red!35] (3,-\i) circle (4pt); }
\foreach \i in {1,...,2} { \fill[red!35] (4,-\i) circle (4pt); }
\node[circle,fill=red,inner sep=2.5pt] at (1,-1) {};

\foreach \i in {3,...,9} { \fill[blue!30] (4,-\i) circle (4pt); }
\node[circle,fill=blue,inner sep=2.5pt] at (4,-6) {};
\foreach \i in {4,...,8} { \fill[blue!30] (3,-\i) circle (4pt); }
\node[circle,fill=blue,inner sep=2.5pt] at (3,-6) {};
\foreach \i in {5,...,7} { \fill[blue!30] (2,-\i) circle (4pt); }
\node[circle,fill=blue,inner sep=2.5pt] at (2,-6) {};
\foreach \i in {6,...,6} { \fill[blue!30] (1,-\i) circle (4pt); }
\node[circle,fill=blue,inner sep=2.5pt] at (1,-6) {};


\draw[dashed,thick] (0.5,-6) -- (4.5,-6);
\node[right] at (4.6,-6) {\small Level 6};

\node at (2.5,0.2) {\small $T=5$};

\node at (1.3,-1.3) {\small 1};
\node at (4.3,-6.3) {\small 2};
\node at (3.3,-6.3) {\small 3};
\node at (2.3,-6.3) {\small 4};
\node at (1.3,-6.3) {\small 5};

\end{scope}

\begin{scope}[xshift=7cm] 

\def\n{4}
\def\H{15}

\foreach \j in {1,...,\n} {
    \node[circle,fill=black,inner sep=1.5pt] (s\j) at (\j,-1) {};
}
\foreach \j in {1,...,3} {
    \draw (s\j) -- (s\the\numexpr\j+1\relax);
}

\foreach \j in {1,...,\n} {
    \foreach \i in {2,...,\H} {
        \node[circle,fill=black,inner sep=1.5pt] (v\i-\j) at (\j,-\i) {};
        \ifnum\i=2
            \draw (s\j) -- (v\i-\j);
        \else
            \draw (v\the\numexpr\i-1\relax-\j) -- (v\i-\j);
        \fi
        \ifnum\i=\H
            \draw (v\i-\j) -- ++(0,-0.5);
        \fi
    }
}

\begin{scope}[on background layer]

\fill[red!30,opacity=0.35,rounded corners=10pt]
    (0.6,-.6) -- (4.4,-0.6) -- (4.4,-3.2) -- (0.6,-6.7) -- cycle;

\fill[blue!30,opacity=0.35,rounded corners=10pt]
    (0.6,-6.8) -- (4.4,-3.3) -- (4.4,-12.8) -- (0.6,-9.2) -- cycle;

\fill[green!35,opacity=0.4,rounded corners=8pt]
    (0.6,-9.3) -- (1.6,-10) -- (0.6,-10.7) -- cycle;

\end{scope}

\foreach \i in {1,...,6} { \fill[red!35] (1,-\i) circle (4pt); }
\foreach \i in {1,...,5} { \fill[red!35] (2,-\i) circle (4pt); }
\foreach \i in {1,...,4} { \fill[red!35] (3,-\i) circle (4pt); }
\foreach \i in {1,...,3} { \fill[red!35] (4,-\i) circle (4pt); }
\node[circle,fill=red,inner sep=2.5pt] at (1,-1) {};

\foreach \i in {4,...,12} { \fill[blue!30] (4,-\i) circle (4pt); }
\node[circle,fill=blue,inner sep=2.5pt] at (4,-8) {};
\foreach \i in {5,...,11} { \fill[blue!30] (3,-\i) circle (4pt); }
\node[circle,fill=blue,inner sep=2.5pt] at (3,-8) {};
\foreach \i in {6,...,10} { \fill[blue!30] (2,-\i) circle (4pt); }
\node[circle,fill=blue,inner sep=2.5pt] at (2,-8) {};
\foreach \i in {7,...,9} { \fill[blue!30] (1,-\i) circle (4pt); }
\node[circle,fill=blue,inner sep=2.5pt] at (1,-8) {};

\foreach \i in {10} { \fill[green!40] (1,-\i) circle (4pt); }
\node[circle,fill=green!60!black,inner sep=2.5pt] at (1,-10) {};

\draw[dashed,thick] (0.5,-10) -- (4.5,-10);
\node[right] at (4.6,-10) {\small Level 10};

\node at (2.5,0.2) {\small $T=6$};

\node at (1.3,-1.3) {\small 1};
\node at (4.3,-8.3) {\small 2};
\node at (3.3,-8.3) {\small 3};
\node at (2.3,-8.3) {\small 4};
\node at (1.3,-8.3) {\small 5};
\node at (1.3,-10.3) {\small 6};

\end{scope}

\begin{scope}[xshift=14cm] 

\def\n{4}
\def\H{15}

\foreach \j in {1,...,\n} {
    \node[circle,fill=black,inner sep=1.5pt] (s\j) at (\j,-1) {};
}
\foreach \j in {1,...,3} {
    \draw (s\j) -- (s\the\numexpr\j+1\relax);
}

\foreach \j in {1,...,\n} {
    \foreach \i in {2,...,\H} {
        \node[circle,fill=black,inner sep=1.5pt] (v\i-\j) at (\j,-\i) {};
        \ifnum\i=2
            \draw (s\j) -- (v\i-\j);
        \else
            \draw (v\the\numexpr\i-1\relax-\j) -- (v\i-\j);
        \fi
        \ifnum\i=\H
            \draw (v\i-\j) -- ++(0,-0.5);
        \fi
    }
}

\begin{scope}[on background layer]

\fill[red!30,opacity=0.35,rounded corners=10pt]
    (0.6,-.6) -- (4.4,-0.6) -- (4.4,-4.2) -- (0.6,-7.7) -- cycle;

\fill[blue!30,opacity=0.35,rounded corners=10pt]
    (0.6,-7.8) -- (4.4,-4.3) -- (4.4,-15.8) -- (0.6,-12.2) -- cycle;

\fill[green!35,opacity=0.4,rounded corners=8pt]
    (0.6,-12.3) -- (2.6,-14.0) -- (0.6,-15.8) -- cycle;

\end{scope}

\foreach \i in {1,...,7} { \fill[red!35] (1,-\i) circle (4pt); }
\foreach \i in {1,...,6} { \fill[red!35] (2,-\i) circle (4pt); }
\foreach \i in {1,...,5} { \fill[red!35] (3,-\i) circle (4pt); }
\foreach \i in {1,...,4} { \fill[red!35] (4,-\i) circle (4pt); }
\node[circle,fill=red,inner sep=2.5pt] at (1,-1) {};

\foreach \i in {5,...,15} { \fill[blue!30] (4,-\i) circle (4pt); }
\node[circle,fill=blue,inner sep=2.5pt] at (4,-10) {};
\foreach \i in {6,...,14} { \fill[blue!30] (3,-\i) circle (4pt); }
\node[circle,fill=blue,inner sep=2.5pt] at (3,-10) {};
\foreach \i in {7,...,13} { \fill[blue!30] (2,-\i) circle (4pt); }
\node[circle,fill=blue,inner sep=2.5pt] at (2,-10) {};
\foreach \i in {8,...,12} { \fill[blue!30] (1,-\i) circle (4pt); }
\node[circle,fill=blue,inner sep=2.5pt] at (1,-10) {};

\foreach \i in {13,...,15} { \fill[green!40] (1,-\i) circle (4pt); }
\node[circle,fill=green!60!black,inner sep=2.5pt] at (1,-14) {};
\fill[green!40] (2,-14) circle (4pt);
\node[circle,fill=green!60!black,inner sep=2.5pt] at (2,-14) {};

\draw[dashed,thick] (0.5,-14) -- (4.5,-14);
\node[right] at (4.6,-14) {\small Level 14};

\node at (2.5,0.2) {\small $T=7$};

\node at (1.3,-1.3) {\small 1};
\node at (4.3,-10.3) {\small 2};
\node at (3.3,-10.3) {\small 3};
\node at (2.3,-10.3) {\small 4};
\node at (1.3,-10.3) {\small 5};
\node at (1.3,-14.3) {\small 6};
\node at (2.3,-14.3) {\small 7};

\end{scope}

\end{tikzpicture}
\caption{Greedy burning on $C_{4,\infty}$ using $T=5$ (left), $T=6$ (middle), and $T=7$ (right) fires.
Shaded regions indicate schematic burn neighborhoods for the first fire and each subsequent round of at most $4$ fires; dashed lines indicate last fully burned levels.}
\label{fig: C_4_infty}
\end{figure}
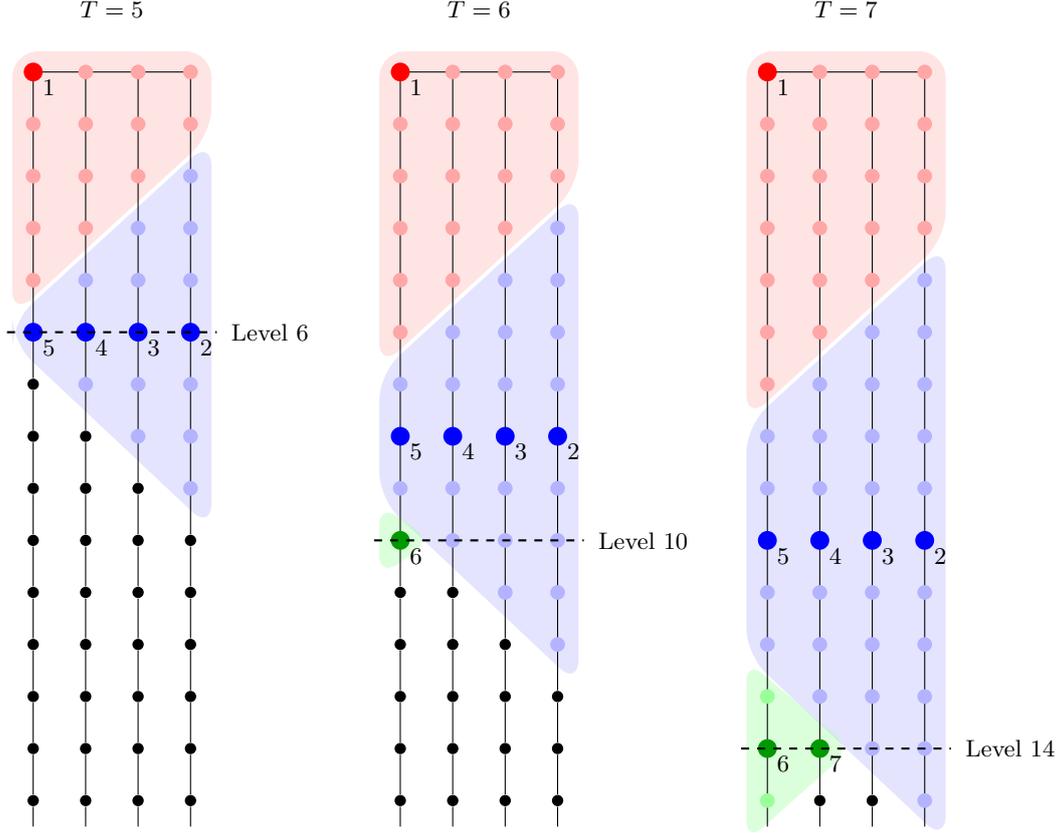

We now approach establishing the BNC in the tooth-dominant regime, when $n \le m$. Unlike in the spine-dominant regime, we do not pursue finding a closed form for the burning number. Instead, we analyze $T_\gr$ directly using a geometric packing argument to establish $T_\gr \le \lceil \sqrt{nm}\rceil$.

One convenient way to view the burning sequence produced by 
\textsc{GreedyComb}$(T;n,m;S)$ is as the projection of the greedy sequence on 
\textsc{GreedyComb}$(T;n,\infty;S)$ onto $C_{n,m}$. 
That is, we first run the greedy process on the infinite comb $C_{n,\infty}$ (which need not succeed), and then project onto $C_{n,m}$ by relocating any fires that would lie below level $m$ to the corresponding leaf vertices. \Cref{fig: C_4_infty} shows three examples of running the greedy process on $C_{4,\infty}$ using $T = 5,6,7$ fires.

Thus $C_{n,m}$ is fully burned in $T$ rounds by \textsc{GreedyComb} if and only if 
\[
m \le M(T),
\]
where
\[
M(T)=M(T,n)=\max\{m \ge 1 : \textsc{GreedyComb}(T;n,m;1)\text{ succeeds}\}.
\]
Equivalently, $M(T)$ is the number of fully burned layers of $C_{n,\infty}$ after $T$ rounds. Consequently,
\[
T_{\gr}(n,m)=\min\{T\ge1 : m\le M(T)\}.
\]
Our task is therefore to determine $M(T)$ via the greedy process on the infinite comb.


The first fire is placed at the top-left spine vertex. After $T$ rounds it burns the entire spine and consumes $T,T-1,\dots,T-n+1$ vertices on the teeth from left to right. Removing this region leaves a disjoint union of paths (a path forest). Moreover, we highlight the vertices covered by this fire comprise a triangular boundary.

We now analyze the greedy process on $C_{n,\infty}$. Since the remaining teeth all have infinite length, the greedy rule selects the tallest available segment at each step. Since the burning radius decreases by one each round, the next $n$ fires are placed at the same height, beginning at the rightmost tooth and moving left one tooth at a time. The union of these $n$ burning regions forms a trapezoid fitting beneath the triangular boundary created by the first fire.

This structure repeats in blocks of $n$ rounds: each block contributes a smaller trapezoid beneath the previous triangular boundary. If $r=(T-1)\bmod n$, the final $r$ fires form a triangular region that fits against the boundary created by the preceding full blocks. (See \Cref{fig: C_4_infty} for the cases $T=5,6,7$.)

So now to determine when \textsc{GreedyComb} can succeed on $C_{n,m}$ using $T$ rounds of fires, then we necessarily need to have $m \le M(T)$ for $M(T)$ the maximal number of fully burned layers of $C_{n,\infty}$ using $T$ fires. Equivalently, we need $nm \le n M(T)$, as these are the number of vertices from layers in $C_{n,\infty}$ fully burned using $T$ fires. So we now pursue finding the closed form of $n M(T)$:

\begin{lemma}
    For $m \ge n$ and $T \ge n$, then 
    \[
    nM(T) = T^2 + (n-2) T + 1 + r(n-r) - n(n-1)
    \]
    for $r = (T-1)\pmod n$.
\end{lemma}

\begin{proof}
    As outlined above: for fixed $T$ we can determine $M(T)$ that constitutes the maximal number of levels fulling burned in the infinite lattice of width $n$. This can be computed as follows:

    \medskip
\noindent
\emph{First fire contribution:} The total vertices burned by the first fire is 
    \[
    T + (T-1) + (T-2) + \ldots + (T-n+1) = nT - \frac{n(n-1)}2.\]
    \medskip
    \noindent
    \emph{Remaining fire contributions.} Since in the infinite lattice the greedy placements place the remaining fires at optimal heights to not overlap with the burning radii of a previous fire, so we have disjoint line segments burned, that then constitute:
    \[
    \sum_{j=2}^T[2(T-j) + 1] = \sum_{s=0}^{T-2}[2s+1] = (T-1)^2
    \]
    vertices burned.
    
    \medskip
    \noindent
    \emph{Removing overhang vertices.} It remains to remove the vertices burned in the overhang regions of the infinite comb. 
    Specifically, we subtract the vertices in the final triangular region that lie below the level of the last fire centers, as well as the portion of the preceding $n$-round trapezoidal region that extends below this same level.
    
    Let $r = (T-1) \pmod n$ and write $T-1 = qn + r$. The final $r$ fires form a triangular region beneath the preceding $n$-fire trapezoid, and each of these fires are placed at the same height level that thus determines the final fully burned layer in the infinite comb. The last fire is placed in tooth $r$ from the left or tooth $n-r+1$ from the right (depending on $q$). The preceding $r-1$ fires in this block burn vertices strictly below this level, contributing $$\sum_{j=1}^{r-1} j = \frac{r(r-1)}{2}$$ vertices beneath the final fully burned level.
    
    Similarly, the lower triangular portion of the previous $n$-fire trapezoid extends below the height of the final fire. This overhang contributes $$\sum_{j=1}^{n-r-1} j = \frac{(n-r)(n-r-1)}{2}$$ additional burned vertices below the last completely burned level. Observe that when $T=n$, there is no such overhang to subtract. Indeed, in this case $r=(T-1)\pmod n = n-1$, so $n-r-1=0$, and the sum above is empty, as expected.

    \medskip
    So adding the total vertices burned by the $T$ fires and then subtracting the additional vertices below the last fully burned level gives us:
    \begin{align*}
        nM(T) &= nT - \frac{n(n-1)}2 + (T-1)^2 - \frac{r(r-1)}2 - \frac{(n-r)(n-r-1)}2 \\
        &= (T-1)^2 + nT + r(n-r)-n(n-1)\\
        &= T^2 + (n-2)T + 1 + r(n-r) - n(n-1).
    \end{align*}
\end{proof}

It follows that to determine whether $C_{n,m}$ can be burned using $T$ fires, one only needs to check if $m \le M(T)$. We first separate a few initial case. 

If $n = 1$, then $C_{1,m} = P_m$ is a single tooth, and \textsc{GreedyComb} then aligns precisely with the standard greedy burning of a path, maximally burning off the top segment of the path using $\lceil \sqrt m\rceil$ fires. 

If $n = 2$, then $C_{2,m} = P_{2m}$ again is actually a path, and so we know  precisely $b(C_{2,m}) = b(P_{2m}) = \lceil \sqrt{2m}\rceil$. As is seen in \Cref{fig: P_16}, a greedy burning sequence can optimally burn a path if always starting each subsequent fire maximally spaced from one end of the path and moving toward the far end. It is not, however, immediately obvious whether or not a successful burning sequence can also be constructed by placing the first fire at a center vertex of $P_{2m}$ and then burning the remaining two outside segments greedily moving outwards. We show this now: 

\begin{lemma}\label{l: n=2}
    If $m \ge n = 2$ and $T^* = \left\lceil \sqrt{2m}\right\rceil$, then $m \le M(T^*)$.
\end{lemma}
\begin{proof}
    This reduces to checking $2M(T^*) - 2m \ge 0$. We now have
    \begin{align*}
        2M(T^*) - 2m = (T^*-1)^2 + 2T^* + r(2-r) -2 -2m= (T^*)^2 -2m + r - 1 \ge 0
    \end{align*}
    where we further note $r = (T^*-1)\mod 2$  yielded $r(2-r) = r$. Now if $r = 1$, then this reduces to checking $(T^*)^2 - 2m \ge 0$ which follows by definition of $T^*$. If $r = 0$, then $T^*-1$ is even so $T^*$ is odd. Since $(T^*)^2 \ge 2m$ then necessarily $(T^*)^2 \ge 2m + 1$. It follows then $m \le M(T^*)$.
\end{proof}

\begin{remark}
    \Cref{l: n=2} shows that a path with an even number of vertices can be optimally burned by placing the first fire at a central vertex. The odd case requires a parity adjustment.

    Let $T$ be optimal for $C_{2,m}=P_{2m}$ and consider the parity of $T-1$. If $T-1$ is even, then the last fire is placed on the first tooth, and the penultimate fire burns one additional vertex on the second tooth in the infinite comb. If $T-1$ is odd, then the last fire is placed on the second tooth, and the penultimate fire instead burns one additional vertex on the first tooth.
    
    In either case, a successful burning of $P_{2m+1}$ is obtained by starting with an optimal greedy burning of $C_{2,m}=P_{2m}$ and placing the additional vertex on the tooth containing the penultimate fire, so that it is consumed in the final round.
\end{remark}

Next, we can combine the previous cases from \Cref{sec: balanced} to show:
\begin{lemma}\label{l: td2}
    If $m \ge n$ and $0 \le m - n \le 2$, then $m \le M(T^*)$ for $T^* = \lceil \sqrt{nm}\rceil$.
\end{lemma}

\begin{proof}
    By \Cref{l: n=m}, then $b(C_{n,n}) = n = T^*$. We note also $r = (T^*-1)\pmod n = n-1$ and so $r(n-r) = n-1$ and so $$n M(T^*) = n^2 + (n-2)n + 1 + n-1 - n(n-1) = n^2 + n^2 - 2n + 1 + n-1-n^2 +n = n^2,$$ so $n \le M(T^*) = n$. 
    
    Next,  by \Cref{l: n_n+1,l: n_n+2}, then $b(C_{n,n+1}) = b(C_{n,n+2}) = n + 1 = T^*$, while we similarly compute $r = (T^*-1)\pmod n = 0$, so that $r(n-r) = 0$. It follows
    $$nM(T^*) = n^2 + n(n+1) + 0 - n(n-1) = n^2 + n^2 + n - n^2 + n = n^2+2n = n(n+2),$$
    so that $m \le M(T^*) = n+2$ for $m \in \{n+1,n+2\}$.
\end{proof} 

\begin{remark}\label{rmk: n+2}
    As seen in \Cref{tab:small diff}, a simple check shows $M(n+2) = n + 6$, so we necessarily have $n+2 = b(C_{n,n+3}) \le b(C_{n,m}) \le T_\gr = n+2$ for $m = n + c$ for $c = 3,4,5,6$, so that also $b(C_{n,m}) = n+2$ for $m$ in this range.
\end{remark}

So it remains to now only consider $m \ge n + 3$.

\begin{lemma}\label{l: td3}
    Let $m > n \ge 3$. Let $T^* = \lceil \sqrt{nm} \rceil$ and $m \ge n+3$. Then $m \le M(T^*)$.
\end{lemma}

\begin{proof}
    It suffices to show that $nM(T^*) - nm \ge 0$. By definition, we have $(T^*)^2 \ge nm$. Since $n \ge 3$, we also have $nm \ge n^2 + 3n > (n+1)^2$, and hence $T^* \ge n+2$. Next, writing $T^*-1 = qn + r$ with $q \ge 1$ and $0 \le r < n$, we note that $r(n-r) \ge 0$.

    Now we have everything needed to go forward: 
    \begin{align*}
        nM(T^*) - nm &= [(T^*)^2 - nm]+ (n-2)T^* + 1 + r(n-r) - n(n-1)\\
        &\ge 0 + (n-2)(n+2) + 1 + 0- n(n-1)\\
        &= n^2-4 + 1-n^2+n\\
        &=n-3\\
        &\ge 0.
    \end{align*}
\end{proof}

This now fully establishes the BNC also is true for the tooth-dominant regime (see \Cref{prop: tooth-dominant}):

\begin{lemma}\label{l: tooth upper}
    If $n \ge m$, then $b(C_{n,m}) \le T^* = \lceil \sqrt{nm}\rceil$.
\end{lemma}
\begin{proof}
    Combining \Cref{l: n=2,l: td2,l: td3}, we have $m \le M(T^*)$ for $T^* = \lceil \sqrt{nm}\rceil$, and so
    \[
    b(C_{n,m}) \le T_\gr = \min\{T \ge 1: m \le M(T^*)\} \le T^*.
    \]
\end{proof}

Now combining this with the spine-dominant regime case, we have completed the proof of the BNC holding for all comb graphs $C_{n,m}$ (see \Cref{thm: bnc}):

\begin{proof}[Proof of \Cref{thm: bnc}]
    Combine \Cref{prop: spine-dominant} and \Cref{l: tooth upper}.
\end{proof}

In the next section, we pursue the remaining bounds relating $b(C_{n,m})$ and $T_\gr$ highlighted in \Cref{prop: tooth-dominant}, where we show $T_\gr$ remains a good approximation for the burning number even if it is no longer an exact match.

\subsubsection{Greedy approximation of the burning number}

\begin{figure}[t]
\centering
\begin{tikzpicture}[scale=0.69]

\begin{scope}[xshift=0cm]

\def\n{3}
\def\H{19}

\foreach \j in {1,...,\n} {
    \node[circle,fill=black,inner sep=1.5pt] (Ls\j) at (1,-\j) {};
}
\foreach \j in {1,2} {
    \draw (Ls\j) -- (Ls\the\numexpr\j+1\relax);
}

\foreach \j in {1,...,\n} {
    \foreach \i in {2,...,\H} {
        \node[circle,fill=black,inner sep=1.5pt] (Lv\j-\i) at (\i,-\j) {};
        \ifnum\i=2
            \draw (Ls\j) -- (Lv\j-\i);
        \else
            \draw (Lv\j-\the\numexpr\i-1\relax) -- (Lv\j-\i);
        \fi
        \ifnum\i=\H
            \draw (Lv\j-\i) -- ++(0.5,0);
        \fi
    }
}

\begin{scope}[on background layer]

\fill[red!30,opacity=0.35,rounded corners=10pt]
    (0.6,-3.4) -- (0.6,-0.6) -- (5.3,-0.6) -- (7.8,-3.4) --  cycle;

\fill[blue!30,opacity=0.35,rounded corners=10pt]
    (5.3,-0.6) -- (16.6,-0.6)-- (14.3,-3.4) -- (7.8,-3.4) -- cycle;

\fill[green!35,opacity=0.4,rounded corners=8pt]
    (17, -0.3) --  (19.7,-3.4) -- (14.3,-3.4) -- cycle;

\end{scope}


\foreach \i in {1,...,7} { \fill[red!35] (\i,-3) circle (4pt); }
\foreach \i in {1,...,6} { \fill[red!35] (\i,-2) circle (4pt); }
\foreach \i in {1,...,5} { \fill[red!35] (\i,-1) circle (4pt); }
\node[circle,fill=red,inner sep=2.5pt] at (1,-3) {};

\foreach \i in {6,...,16} { \fill[blue!30] (\i,-1) circle (4pt); }
\node[circle,fill=blue,inner sep=2.5pt] at (11,-1) {};
\foreach \i in {7,...,15} { \fill[blue!30] (\i,-2) circle (4pt); }
\node[circle,fill=blue,inner sep=2.5pt] at (11,-2) {};
\foreach \i in {8,...,14} { \fill[blue!30] (\i,-3) circle (4pt); }
\node[circle,fill=blue,inner sep=2.5pt] at (11,-3) {};

\foreach \i in {15,...,19} { \fill[green!40] (\i,-3) circle (4pt); }
\node[circle,fill=green!60!black,inner sep=2.5pt] at (17,-3) {};
\foreach \i in {16,...,18} { \fill[green!40] (\i,-2) circle (4pt); }
\node[circle,fill=green!60!black,inner sep=2.5pt] at (17,-2) {};
\foreach \i in {17,...,17} { \fill[green!40] (\i,-1) circle (4pt); }
\node[circle,fill=green!60!black,inner sep=2.5pt] at (17,-1) {};

\draw[dashed,thick] (17,-0.5) -- (17,-3.5);
\node[above] at (17,-0.3) {\small Level 17};

\node at (10,0) {\small Greedy (corner)};

\node at (1.3,-3.3) {\small 1};
\node at (11.3,-1.3) {\small 2};
\node at (11.3,-2.3) {\small 3};
\node at (11.3,-3.3) {\small 4};
\node at (17.3,-3.3) {\small 5};
\node at (17.3,-2.3) {\small 6};
\node at (17.3,-1.3) {\small 7};

\end{scope}

\begin{scope}[yshift=-4.5cm]

\def\n{3}
\def\H{19}

\foreach \j in {1,...,\n} {
    \node[circle,fill=black,inner sep=1.5pt] (Rs\j) at (1,-\j) {};
}
\foreach \j in {1,2} {
    \draw (Rs\j) -- (Rs\the\numexpr\j+1\relax);
}

\foreach \j in {1,...,\n} {
    \foreach \i in {2,...,\H} {
        \node[circle,fill=black,inner sep=1.5pt] (Rv\j-\i) at (\i,-\j) {};
        \ifnum\i=2
            \draw (Rs\j) -- (Rv\j-\i);
        \else
            \draw (Rv\j-\the\numexpr\i-1\relax) -- (Rv\j-\i);
        \fi
        \ifnum\i=\H
            \draw (Rv\j-\i) -- ++(0.5,0);
        \fi
    }
}

\begin{scope}[on background layer]

\fill[red!30,opacity=0.35,rounded corners=10pt]
    (0.6,-3.4) -- (0.6,-0.6) -- (6.3,-0.6) -- (7.8,-2) -- (6.3,-3.4) --  cycle;

\fill[blue!30,opacity=0.35,rounded corners=10pt]
    (6.3,-0.6) -- (17.4,-0.6) -- (17.4,-1.5)-- (14.3,-1.5) -- (15.6,-3.4) -- (6.3,-3.4) -- (7.8,-2) -- cycle;

\fill[green!35,opacity=0.4,rounded corners=8pt]
    (17.4, -0.6) --  (18.3,-0.6)  -- (19.5,-2)-- (18.3,-3.4) -- (15.5,-3.4) -- (14.4,-1.5)-- (17.4,-1.5) --  cycle;

\end{scope}


\foreach \i in {1,...,6} { \fill[red!35] (\i,-1) circle (4pt); }
\foreach \i in {1,...,7} { \fill[red!35] (\i,-2) circle (4pt); }
\foreach \i in {1,...,6} { \fill[red!35] (\i,-3) circle (4pt); }
\node[circle,fill=red,inner sep=2.5pt] at (1,-2) {};

\foreach \i in {7,...,17} { \fill[blue!30] (\i,-1) circle (4pt); }
\node[circle,fill=blue,inner sep=2.5pt] at (12,-1) {};
\foreach \i in {8,...,14} { \fill[blue!30] (\i,-2) circle (4pt); }
\node[circle,fill=blue,inner sep=2.5pt] at (11,-2) {};
\foreach \i in {7,...,15} { \fill[blue!30] (\i,-3) circle (4pt); }
\node[circle,fill=blue,inner sep=2.5pt] at (11,-3) {};

\foreach \i in {16,...,18} { \fill[green!40] (\i,-3) circle (4pt); }
\node[circle,fill=green!60!black,inner sep=2.5pt] at (17,-3) {};
\foreach \i in {15,...,19} { \fill[green!40] (\i,-2) circle (4pt); }
\node[circle,fill=green!60!black,inner sep=2.5pt] at (17,-2) {};
\foreach \i in {18,...,18} { \fill[green!40] (\i,-1) circle (4pt); }
\node[circle,fill=green!60!black,inner sep=2.5pt] at (18,-1) {};

\draw[dashed,thick] (18,-0.5) -- (18,-3.5);
\node[above] at (18,-0.3) {\small Level 18};

\node at (10,0) {\small Greedy (center)};

\node at (1.3,-2.3) {\small 1};
\node at (12.3,-1.3) {\small 2};
\node at (11.3,-3.3) {\small 3};
\node at (11.3,-2.3) {\small 4};
\node at (17.3,-2.3) {\small 5};
\node at (17.3,-3.3) {\small 6};
\node at (18.3,-1.3) {\small 7};

\end{scope}

\end{tikzpicture}
\caption{\textsc{GreedyComb} burning sequence on $C_{3,\infty}$ using $T=7$ fires. First fire placed at the top spine vertex, with first fire placed at left corner (top) or the center (bottom) of the spine.}
\label{fig:left_vs_center}
\end{figure}
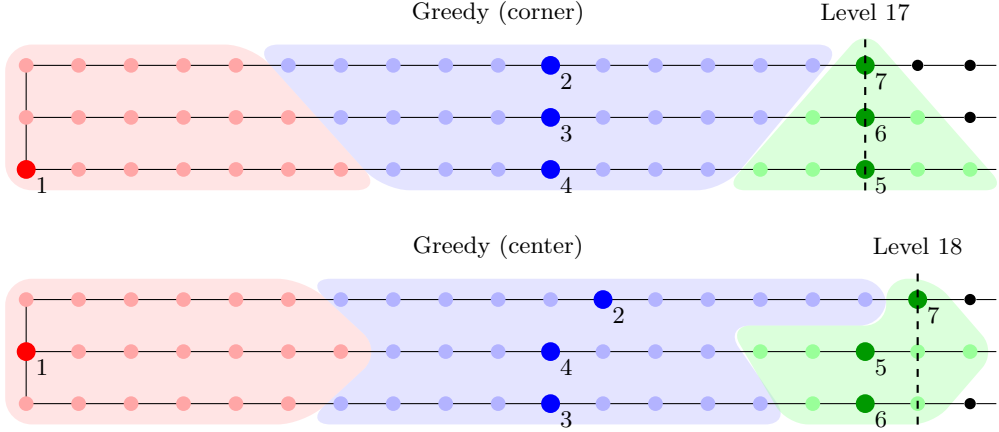

Unlike in the spine-dominant case, we no longer claim that 
\textsc{GreedyComb}$(T_\gr; n,m; 1)$ produces a minimal burning sequence for 
$C_{n,m}$ when $n \le m$. In fact, it no longer holds  that $T_\gr^{(1)}$ is 
minimal among the values $T_\gr^{(S)}$. \Cref{fig:left_vs_center} shows that 
$C_{3,18}$ can be burned with $T_\gr^{(2)} = 7$ fires, while 
$T_\gr^{(1)} = 8$. 

\begin{remark}
    More generally, among the sequences 
\textsc{GreedyComb}$(T;n,m;S)$, we observe (without proving here) that the number of layers burned increases as the 
first and last fires are placed on teeth that are closer together. We additionally note that while the number of totally burned layers may differ, the multiset of collections of fires per teeth remain invariant for each other greedy algorithm; as seen in \Cref{fig:left_vs_center} the three teeth contain the fires $\{1,4,5\}, \{2,7\}, \{3,6\}$, but located on different teeth depending on which tooth receives the first fire. In fact, it can be shown the tooth with the first fire will always contain the fires placed at times $n+1,n+2,3n+1,3n+2, 5n+1,5n+2,\ldots$ while the tooth that receives the second fire then receiving the fires in rounds $2n+1,2n+2,4n+1,4n+2,\ldots$, with the tooth receiving the third fire also receiving fires in rounds $2n,2n+3, 4n,4n+3, \ldots$, and so on.
\end{remark}

We do not attempt to determine the exact burning number in the 
tooth-dominant regime. As already seen in \Cref{prop: n_n+3}, increasing from 
$m = n+2$ to $m = n+3$ significantly complicates the analysis (cf. 
\Cref{l: n_n+2}). Nevertheless, we conjecture that a greedy algorithm always 
yields an optimal solution:

\begin{conjecture}\label{conj: tooth greedy}
If $n \le m$, then
\[
b(C_{n,m}) = \min\left\{T_\gr^{(S)} : S = 1,2,\ldots,n\right\}.
\]
\end{conjecture}

Next, to establish the refined upper and lower bounds on $b(C_{n,m})$ in \Cref{prop: tooth-dominant}, we first work on the upper bound:

\begin{lemma}\label{l: tooth upper sharp}
    For $n \ge m$, then $T_\gr \ge \left\lceil \sqrt{nm}\right\rceil - n/2$.
\end{lemma}
\begin{proof}
    Recall:
    \[
    nM(T) = T^2 + (n-2)T + 1 + r(n-r) - n(n-1).
    \]
    Since $r(n-r) \le n^2/4$ for any remainder $0 \le r < n$. We can thus form an upper bound on $M(T)$ using $M_0(T)$ defined by replacing $r(n-r)$ with $n^2/4$:
    \[
    n M_0(T) = T^2 + (n-2)T + 1 + n - \frac34 n^2.
    \]
    By defining an \textit{upper} bound on $M(T)$ then we can find a \textit{lower} bound for $T_\gr$:
    \begin{equation*}
        T_{\gr} = \min\{T \ge 1: m \le M(T)\} \ge T_0 := \min\{T \ge 1: m \le M_0(T)\}.
    \end{equation*}
We can compute $T_0$ by solving: $nm \le n M_0(T)$ that yields $T_0 = \lceil r^*\rceil$, where $r^*$ is the right-most root of a quadratic polynomial:
\begin{equation*}
    r^* = \frac{-(n-2) + \sqrt{(n-2)^2 + 4(nm + \frac34 n^2 -n-1)}}2 = \sqrt{nm + n(n-2)} +1 - \frac{n}2.
\end{equation*}
It follows
\begin{equation*}
    T_{\gr} \ge T_0 \ge r^* \ge    \sqrt{nm} + 1 - \frac{n}2 \ge \left\lceil \sqrt{nm}\right\rceil- \frac{n}2.
\end{equation*}
\end{proof}

Hence, when $n \le m$, we can compute $T_{\gr}$ as follows. Let 
$T^* = \lceil \sqrt{nm} \rceil$. Since $T_{\gr} \le T^*$, we first check whether 
$nM(T^*-1) \ge nm$. If not, then $T_{\gr} = T^*$. Otherwise, $T_{\gr} < T^*$, and by \Cref{l: tooth lower} we know
\[
T_{\gr} \in 
\left[T^*- \frac{n}{2} , T^* \right].
\]
Moreover, the function $T \mapsto nM(T)$ is increasing in this interval, so once the inequality $nM(T) \ge nm$ holds, it holds for all larger $T$. So we may perform a binary search on this interval, halving it at each step until the minimal such $T$ is found. We summarize this as follows: 

\begin{proposition}\label{prop: tooth runtime}
If $n \le m$, let $T^* = \lceil \sqrt{nm} \rceil$. Then $T_{\gr}$ can be found in $O(\log n)$ time by performing a binary search over $[T^* - n/2, T^*]$ to identify the smallest $T$ with $M(T) \ge m$.
\end{proposition}

\begin{remark}
    \Cref{prop: tooth runtime} was used to efficiently compute each gap $T^* - T_\gr$ for the $n \le m$ cases shown in \Cref{fig:b_comparison}, where $T^* = \lceil \sqrt{nm} \rceil$ is the BNC bound. As illustrated in \Cref{fig:b_comparison}, the observed gap is significantly smaller than the worst-case bound of $n/2$. For $1 \le n \le m \le 5{,}000$, the maximal gap is $417$, attained when $m=5{,}000$ and $2{,}182 \le n \le 2{,}266$, which is well below $n/2$ and closer to $n/5$ in magnitude.
\end{remark}

\begin{remark}
    In combination with \Cref{conj: tooth greedy}, placing the first fire closer to where the final fire is placed can decrease the necessary number of rounds needed for a greedy algorithm to succeed. In particular, we note that a an initial center fire will fully cover the top $\floor{n/2}$ levels of $C_{n,\infty}$. So we could further conjecture if $b(C_{n,m})$ matches one of these greedy algorithm produced burning sequences, then we would actually have $b(C_{n,m}) = \lceil \sqrt{nm}\rceil - d_{n,m}$ for $d_{n,m} \in [0,n]$.
\end{remark}

Next, we pursue establishing the lower bound on $b(C_{n,m})$ seen in \Cref{prop: tooth-dominant}:

\begin{lemma}\label{l: tooth lower}
    If $n \le m$, then $b(C_{n,m}) \ge \sqrt{nm/2}$.
\end{lemma}

\begin{proof}
    As already noted, a lower bound on $b = b(C_{n,m})$ is established using $\hat b = \hat b(C_{n,m})$, where we recall $\hat b = \min \{k: \hat b_{k-1} \le k\}$ for $\hat b_r$ the optimal sphere packing on $C_{n,m}$ using uniform radius. So $\hat b$ involves finding the first instance one can cover the graph using $r$ balls of radius $r-1$. As seen in \Cref{l: hat br}, we now further recall: $\hat b_r = A_r n + B_r$ where $A_r = \floor{m/(2r+1)} \le m/(2r+1)$ and $B_r \le n$, so that
    \[
    \hat b_r \le u_r := \frac{nm}{2r+1} + n.
    \]
    Hence, we can now define a further lower bound on $\hat b$ of $\tilde b = \tilde b(C_{n,m})$ using an upper bound on $\hat b_r$ by defining:
    \[
    \tilde b = \min\left\{ r :  u_{r-1} \le r\right\} \le \hat b \le b.
    \]
    We can now compute $\tilde b$ by optimally solving $$u_{r-1} = \frac{nm}{2r-1} + n \le r,$$
    which yields again $\tilde b = \lceil r^{**} \rceil $ for $r^{**}$ a right-most root of a quadratic polynomial, so that
    \begin{align*}
    \tilde b = \left\lceil \frac1 4 (\sqrt{8 m n + (2n-1)^2} + 2 n + 1)\right\rceil \ge \frac12 \sqrt{n(2m+n)} \ge \sqrt{\frac{nm}2}
    \end{align*}
\end{proof}

We now have the remaining bounds provided in \Cref{prop: tooth-dominant}.

\begin{proof}[Proof of \Cref{prop: tooth-dominant}]
    Using \Cref{l: tooth upper,l: tooth lower}, we have
    \[
    \sqrt{nm/2} \le b(C_{n,m}) \le T_\gr = \lceil \sqrt{nm}\rceil - d_{n,m},
    \]
    where $d_{n,m} \in [0,n/2]$.
\end{proof}

Moreover, although we no longer claim optimality for our particular greedy algorithm in the tooth-dominant regime, we now show that $T_\gr$ at worst is a good approximation:

\begin{proposition}\label{prop: approx}
    If $n \le m$, then $T_{\gr}$ is at most a $\sqrt{2}$-approximation of $b(C_{n,m})$ asymptotically in $nm$. 
\end{proposition}
\begin{proof}
    Using \Cref{l: tooth upper,l: tooth lower}, we have 
    \[
    \sqrt{nm/2} \le b(C_{n,m}) \le T_{\gr} \le  \lceil \sqrt{nm} \rceil \le \sqrt{nm} + 1,
    \]
    so that
    \begin{equation*}
    \frac{T_{\gr}}{b(C_{n,m})} \le \frac{\sqrt{nm} + 1}{\sqrt{nm/2}} = \sqrt 2 \cdot \left(1 + \frac1{\sqrt{nm}}\right) = \sqrt 2 \cdot (1 + o_{nm}(1)).
    \end{equation*} 
\end{proof}

Using \textsc{GreedyComb} with $T \ge n$, the first fire always burns the entire spine, leaving a disjoint union of paths, that is, a path forest. Thus \Cref{prop: approx} may be compared with the general result that the greedy algorithm yields a $3/2$-approximation on arbitrary path forests \cite{bonato2019bounds}; see also \cite{bonato2019approximation} for an overview of approximation algorithms and \cite{tan2020graph} for further asymptotic results on burning path forests.

In contrast, the rigid layered structure of $C_{n,m}$  allows us to exploit additional geometric constraints that are absent in general path forests. As a result, when $n \le m$ and $nm \ge 272$, we improve upon the $3/2$-approximation ratio. (When $n \ge m$, we already have $b(C_{n,m}) = T_\gr$.)

\bibliographystyle{alpha} 
\bibliography{references}

\newcommand{\etalchar}[1]{$^{#1}$}
\begin{thebibliography}{BEKM21}

\bibitem[AV05]{Abert_Virag_2005}
Miklós Abért and Bálint Virág.
\newblock Dimension and randomness in groups acting on rooted trees.
\newblock {\em Journal of the American Mathematical Society}, 18(1):157–192, 2005.

\bibitem[BBB{\etalchar{+}}23]{BastideBonamyBonatoCharbitKamaliPierronRabie2023}
Paul Bastide, Marthe Bonamy, Anthony Bonato, Pierre Charbit, Shahin Kamali, Th\'eo Pierron, and Mika\"el Rabie.
\newblock Improved pyrotechnics: Closer to the burning graph conjecture.
\newblock {\em Electronic Journal of Combinatorics}, 30(4):P4.2, 2023.

\bibitem[BBJ{\etalchar{+}}17]{BessyBonatoJanssenRautenbachRoshanbin2017}
St{\'e}phane Bessy, Anthony Bonato, Jeannette Janssen, Dieter Rautenbach, and Elham Roshanbin.
\newblock Burning a graph is hard.
\newblock {\em Discrete Applied Mathematics}, 232:73--87, 2017.

\bibitem[BBJ{\etalchar{+}}18]{bessy2018bounds}
St{\'e}phane Bessy, Anthony Bonato, Jeannette Janssen, Dieter Rautenbach, and Elham Roshanbin.
\newblock Bounds on the burning number.
\newblock {\em Discrete Applied Mathematics}, 235:16--22, 2018.

\bibitem[BC25]{blanc2025random}
Guillaume Blanc and Alice Contat.
\newblock Random burning of the {E}uclidean lattice.
\newblock {\em arXiv preprint arXiv:2509.02562}, 2025.

\bibitem[BDPX09]{boyd2009fastest}
Stephen Boyd, Persi Diaconis, Pablo Parrilo, and Lin Xiao.
\newblock Fastest mixing {M}arkov chain on graphs with symmetries.
\newblock {\em SIAM Journal on Optimization}, 20(2):792--819, 2009.

\bibitem[BEKM21]{bonato2021improved}
Anthony Bonato, Sean English, Bill Kay, and Daniel Moghbel.
\newblock Improved bounds for burning fence graphs.
\newblock {\em Graphs and Combinatorics}, 37(6):2761--2773, 2021.

\bibitem[BGNP26]{barrett2026achievable}
Jordan Barrett, Karen Gunderson, JD~Nir, and Pawe{\l} Pra{\l}at.
\newblock Achievable burning densities of growing grids.
\newblock {\em arXiv preprint arXiv:2601.14151}, 2026.

\bibitem[BGS25]{BorgaGwynneSun2025}
Jacopo Borga, Ewain Gwynne, and Xin Sun.
\newblock Permutons, meanders, and {SLE}-decorated {L}iouville quantum gravity.
\newblock {\em Journal of the European Mathematical Society}, 2025.
\newblock Published online 12 June 2025, to appear.

\bibitem[BHSY23]{borga2023baxter}
Jacopo Borga, Nina Holden, Xin Sun, and Pu~Yu.
\newblock Baxter permuton and {L}iouville quantum gravity.
\newblock {\em Probability Theory and Related Fields}, 186(3–4):1225--1273, 2023.

\bibitem[BJR16]{bonato2016burn}
Anthony Bonato, Jeannette Janssen, and Elham Roshanbin.
\newblock How to burn a graph.
\newblock {\em Internet Mathematics}, 12(1-2):85--100, 2016.

\bibitem[BK19]{bonato2019approximation}
Anthony Bonato and Shahin Kamali.
\newblock Approximation algorithms for graph burning.
\newblock In {\em International Conference on Theory and Applications of Models of Computation}, pages 74--92. Springer, 2019.

\bibitem[BL19]{bonato2019bounds}
Anthony Bonato and Thomas Lidbetter.
\newblock Bounds on the burning numbers of spiders and path-forests.
\newblock {\em Theoretical Computer Science}, 794:12--19, 2019.

\bibitem[DDS{\etalchar{+}}18]{das2018burning}
Sandip Das, Subhadeep~Ranjan Dev, Arpan Sadhukhan, Uma~Kant Sahoo, and Sagnik Sen.
\newblock Burning spiders.
\newblock In {\em Conference on Algorithms and Discrete Applied Mathematics}, pages 155--163. Springer, 2018.

\bibitem[DEP25]{devroye2025burning}
Luc Devroye, Austin Eide, and Pawe{\l} Pra{\l}at.
\newblock Burning random trees.
\newblock {\em Electronic Communications in Probability}, 30:1--9, 2025.

\bibitem[DIMP23]{das2023burning}
Sandip Das, Sk~Samim Islam, Ritam~M Mitra, and Sanchita Paul.
\newblock Burning a binary tree and its generalization.
\newblock {\em arXiv preprint arXiv:2308.02825}, 2023.

\bibitem[GHS{\etalchar{+}}09]{Gamburd_Hoory_Shahshahani_Shalev_Virag_2009}
A~Gamburd, S~Hoory, M~Shahshahani, A~Shalev, and B~Virág.
\newblock On the girth of random {C}ayley graphs.
\newblock {\em Random Structures \& Algorithms}, 35(1):100–117, 2009.

\bibitem[HKT20]{hiller2020burning}
Michaela Hiller, Arie~MCA Koster, and Eberhard Triesch.
\newblock On the burning number of $p$-caterpillars.
\newblock In {\em Graphs and Combinatorial Optimization: From Theory to Applications: CTW2020 Proceedings}, pages 145--156. Springer, 2020.

\bibitem[LHH20]{liu2020burning}
Huiqing Liu, Xuejiao Hu, and Xiaolan Hu.
\newblock Burning number of caterpillars.
\newblock {\em Discrete Applied Mathematics}, 284:332--340, 2020.

\bibitem[LL16]{land2016upper}
Max~R Land and Linyuan Lu.
\newblock An upper bound on the burning number of graphs.
\newblock In {\em International Workshop on Algorithms and Models for the Web-Graph}, pages 1--8. Springer, 2016.

\bibitem[LLD24]{lindquist2024generalizing}
Neil Lindquist, Piotr Luszczek, and Jack Dongarra.
\newblock Generalizing random butterfly transforms to arbitrary matrix sizes.
\newblock {\em ACM Transactions on Mathematical Software}, 50(4):1--23, 2024.

\bibitem[MPR17]{mitsche2017burning}
Dieter Mitsche, Pawe{\l} Pra{\l}at, and Elham Roshanbin.
\newblock Burning graphs: a probabilistic perspective.
\newblock {\em Graphs and Combinatorics}, 33(2):449--471, 2017.

\bibitem[MPR18]{mitsche2018burning}
Dieter Mitsche, Pawe{\l} Pra{\l}at, and Elham Roshanbin.
\newblock Burning number of graph products.
\newblock {\em Theoretical Computer Science}, 746:124--135, 2018.

\bibitem[Mur24]{murakami2024burning}
Yukihiro Murakami.
\newblock The burning number conjecture is true for trees without degree-2 vertices.
\newblock {\em Graphs and Combinatorics}, 40(4):82, 2024.

\bibitem[NT24]{norin2024burning}
Sergey Norin and J{\'e}r{\'e}mie Turcotte.
\newblock The burning number conjecture holds asymptotically.
\newblock {\em Journal of Combinatorial Theory, Series B}, 168:208--235, 2024.

\bibitem[Par95]{Pa95}
D.~Stott Parker.
\newblock Random butterfly transformations with applications in computational linear algebra.
\newblock {\em Tech. rep., UCLA}, 1995.

\bibitem[PM24]{P24}
John Peca-Medlin.
\newblock Distribution of the number of pivots needed using {G}aussian elimination with partial pivoting on random matrices.
\newblock {\em Ann. Appl. Probab.}, 34(2):2294--2325, 2024.

\bibitem[PM25]{peca2025horton}
John Peca-Medlin.
\newblock The {H}orton-{S}trahler number of butterfly trees.
\newblock {\em arXiv preprint arXiv:2509.11384}, 2025.

\bibitem[PMT23]{PT23}
John Peca-Medlin and Thomas Trogdon.
\newblock Growth factors of random butterfly matrices and the stability of avoiding pivoting.
\newblock {\em SIAM J. Matrix Anal. Appl.}, 44(3):945--970, 2023.

\bibitem[PMZ24]{PZ24}
John Peca-Medlin and Chenyang Zhong.
\newblock On the longest increasing subsequence and number of cycles of butterfly permutations.
\newblock {\em arXiv preprint arXiv:2410.20952}, 2024.

\bibitem[PMZ25]{PZ25}
John Peca-Medlin and Chenyang Zhong.
\newblock Heights of butterfly trees.
\newblock {\em arXiv preprint arXiv:2507.04505}, 2025.

\bibitem[Tro19]{Tr19}
Thomas Trogdon.
\newblock On spectral and numerical properties of random butterfly matrices.
\newblock {\em Applied Math. Letters}, 95(4):48--58, September 2019.

\bibitem[TT20]{tan2020graph}
Ta~Sheng Tan and Wen~Chean Teh.
\newblock Graph burning: tight bounds on the burning numbers of path forests and spiders.
\newblock {\em Applied Mathematics and Computation}, 385:125447, 2020.

\end{thebibliography}

\appendix
\section{Random graph products of paths}\label{sec: random}

Using \Cref{prop: spine-dominant,prop: tooth-dominant}, we obtain the sharp scaling regimes for comb graphs.

\begin{corollary}\label{cor: scaling}
For $n,m \ge 1$,
\[
    b(C_{n,m}) =
    \begin{cases}
        \Theta(\sqrt n), & m \le \sqrt n,\\[4pt]
        \Theta(m), & \sqrt n \le m \le n,\\[4pt]
        \Theta(\sqrt{nm}), & n \le m.
    \end{cases}
\]
\end{corollary}

We will compare this behavior to that of Cartesian and strong products of paths.

\begin{theorem}[{\cite{mitsche2017burning,mitsche2018burning}}]
\label{thm: grid products}
For $n\ge m=m(n)$,
\[
b(P_m \square P_n)=
\begin{cases}
(1+o(1))\sqrt[3]{\frac32mn}, & n \ge m=\omega(\sqrt n),\\[6pt]
\Theta(\sqrt n), & m=O(\sqrt n),
\end{cases}
\]
and
\[
b(P_m \boxtimes P_n)=
\begin{cases}
(1+o(1))\sqrt[3]{\frac34mn}, & m=\omega(\sqrt n),\\[6pt]
\Theta(\sqrt n), & m=O(\sqrt n).
\end{cases}
\]
\end{theorem}


We can now construct random graph products where $n,m$ are random variables, so that the above piecewise structure is all that is needed to establish limiting distributions. Throughout, $X_k \Rightarrow X$ denotes convergence in distribution as $k \to \infty$.

\begin{theorem}\label{thm: random comb - long}
Let $U_k \sim \Unif\{0,1,\ldots,k\}$ and define $X_k = 2^{U_k}$ and $Y_k = 2^{k-U_k}$. Then
\begin{align*}
\frac1k \log_2 b(C_{X_k,Y_k}) &\Rightarrow f(U), \\
\frac1k \log_2 \min\{b(C_{X_k,Y_k}),b(C_{Y_k,X_k})\}
&\Rightarrow g(U), \\
\frac1k \log_2 b(P_{X_k}\square P_{Y_k})
&\Rightarrow h(U), \\
\frac1k \log_2 b(P_{X_k}\boxtimes P_{Y_k})
&\Rightarrow h(U),
\end{align*}
where $U \sim \Unif(0,1)$, $g(x)=\min\{f(x),f(1-x)\}$, and
\[
f(x)=
\begin{cases}
1/2, & 0\le x\le 1/2,\\
1-x, & 1/2\le x\le 2/3,\\
x/2, & 2/3\le x\le 1,
\end{cases}
\qquad
h(x)=
\begin{cases}
(1-x)/2, & 0\le x\le 1/3,\\
1/3, & 1/3\le x\le 2/3,\\
x/2, & 2/3\le x\le 1.
\end{cases}
\]
\end{theorem}

\begin{proof}
We prove the comb case; the other cases follow by the same argument using \Cref{thm: grid products}. 
Let $X_k=2^{U_k}$ and $Y_k=2^{k-U_k}$. By \Cref{cor: scaling}, 
\[
\log_2 b(C_{X_k,Y_k})
=
\begin{cases}
\frac12 \log_2(X_kY_k) + o(k), & X_k \le Y_k,\\[4pt]
\log_2 Y_k + o(k), & \sqrt{X_k} \le Y_k \le X_k,\\[4pt]
\frac12 \log_2 X_k + o(k), & Y_k \le \sqrt{X_k}.
\end{cases}
\]
Substituting $X_k=2^{U_k}$ and $Y_k=2^{k-U_k}$ and dividing by $k$ gives
\[
\frac1k \log_2 b(C_{X_k,Y_k})
=
f\!\left(\frac{U_k}{k}\right) + o(1).
\]
Since $U_k/k \Rightarrow U\sim\Unif(0,1)$ and $f$ is continuous, the Continuous Mapping Theorem yields
\[
\frac1k \log_2 b(C_{X_k,Y_k})
\Rightarrow
f(U).
\]
\end{proof}

\begin{figure}[t]
    \centering
    \includegraphics[width=0.6\linewidth]{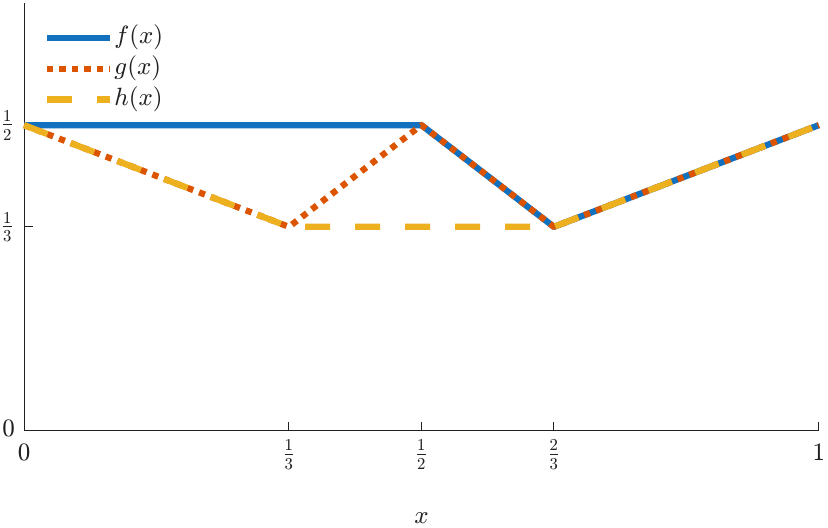}
    \caption{Piecewise functions $f(x)$, $g(x)$, and $h(x)$.}
    \label{fig:piecewise}
\end{figure}

\Cref{fig:piecewise} shows that
\[
h(x) \le g(x) \le f(x), \qquad x\in[0,1],
\]
mirroring the deterministic inequalities
\[
b(P_n \boxtimes P_m)
\le
b(P_n \square P_m)
\le
\min\{b(C_{n,m}),b(C_{m,n})\}
\le
b(C_{n,m}).
\]

\begin{remark}
\Cref{thm: random comb - long} extends naturally. 
If $(U,V)$ is any random vector in $[0,1]^2$ and $a > 1$ with
\[
n = \lfloor a^{ kU }\rfloor, 
\qquad 
m = \lfloor a^{ kV }\rfloor,
\]
then \Cref{cor: scaling} implies
\[
\frac{1}{k}\log_a b(C_{n,m})
\Rightarrow
F(U,V),
\]
where $F$ is obtained by applying the three scaling regimes of \Cref{cor: scaling} to $(U,V)$. Theorem~\ref{thm: random comb - long} corresponds to the special case $U\sim\Unif(0,1)$, $V=1-U$,  and $a = 2$.
\end{remark}

In contrast, several previously studied random graph models exhibit degenerate limiting behavior under similar scaling, converging to constants. For example,
\[
\frac1{\log n}\log b(\operatorname{GW}_n)
\xrightarrow[n\to\infty]{\mathbb P} \frac13,
\qquad
\frac1{\log n}\log b(G(n,p))
\xrightarrow[n\to\infty]{\mathbb P} 0,
\]
for $\operatorname{GW}_n$ a critical Galton--Watson tree conditioned on $n$ vertices and $G(n,p)$ an Erd\H{o}s--R\'enyi graph with sufficiently large edge density \cite{devroye2025burning,mitsche2017burning}. 

This compares instead to $f(U), g(U), h(U)$ from \Cref{thm: random comb - long}, that have full support on $[1/3,1/2]$ and so each comprise nondegenerate distributions. In fact, $g(U) \sim \Unif([1/3,1/2])$ while each of $f(U)$ and $h(U)$ are mixtures of $\Unif([1/3,1/2])$ and Dirac masses at either 1/2 or 1/3.

\section{Burning number of degree-constrained spanning trees}\label{sec: trees}

As observed in \cite{bonato2016burn}, determining the burning number of a graph is equivalent to determining the minimum burning number over all of its spanning trees. Thus the BNC reduces to the following statement:

\begin{conjecture}
If $T$ is a tree on $n$ vertices, then 
\[
b(T) \le \lceil \sqrt{n} \rceil.
\]
\end{conjecture}

Although computing the burning number of a graph is NP-complete, the problem remains NP-complete even when restricted to trees \cite{BessyBonatoJanssenRautenbachRoshanbin2017}. In particular, imposing structural constraints such as bounded degree does not make the problem algorithmically simpler in general.

\subsection{Grids versus combs}

In comparing the burning number of the grid $P_n \,\square\, P_m$ with that of the comb, we recall from \Cref{cor: scaling} and \Cref{thm: grid products} that the comb captures the correct burning scale of the lattice in the regime $m \le \sqrt{n}$. However, when $m = \omega(\sqrt{n})$, the comb no longer reflects the correct lattice-scale behavior. This motivates the study of alternative spanning trees that better preserve two-dimensional spread while maintaining degree constraints.

\Cref{fig:8x8_comparison} compares the $8 \times 8$ grid $P_8 \square P_8$ with several spanning trees. Unlike the comb, the lattice permits simultaneous horizontal and vertical spread, allowing fire to expand in area rather than along essentially one-dimensional channels.

The top  of \Cref{fig:8x8_comparison} exhibit a burning sequence using six rounds, together with a spanning tree realizing the same burning schedule. This establishes $b(P_8 \square P_8) \le 6$. To show equality, suppose instead that five rounds suffice. A case analysis on the placement of the first fire shows this is impossible. If the first fire is placed near the center, its radius-four burn leaves four disjoint corner regions, each containing at least three unburned vertices. The remaining four fires, one of which has radius zero, cannot cover all four regions. If the first fire is placed away from the center, its effective coverage decreases, and its boundary leaves triangular uncovered regions near both corners and boundary vertices. A similar counting argument shows five fires cannot cover all such regions. We leave the remaining details establishing $b(P_8 \square P_8) = 6$ to the motivated reader.

In contrast, the comb $C_{8,8}$ requires eight rounds to burn by \Cref{l: n=m}. The bottleneck along the spine restricts lateral fire propagation and slows global coverage.

\subsection{Simple butterfly trees}

We now consider simple butterfly trees, a recursively defined family of degree-constrained spanning trees of rectangular grids. These trees have maximum degree three. 

Simple butterfly trees arise naturally as binary search trees of recursively defined permutations known as butterfly permutations (see \cite{P24,PZ24}). Butterfly permutations appear as permutation factors in Gaussian elimination with partial pivoting applied to butterfly matrices, a recursive family within the DFT class with wide applications in numerical linear algebra (e.g., \cite{Pa95,Tr19,PT23,lindquist2024generalizing}).

Butterfly permutations form a separable $2$-Sylow subgroup of the symmetric group on $N=2^n$ elements. Separable permutations also appear in mathematical physics and random geometry, including connections to Liouville quantum gravity and Schramm–Loewner evolutions \cite{borga2023baxter,BorgaGwynneSun2025}. Sylow subgroups more broadly arise in group actions on rooted trees \cite{Abert_Virag_2005}, with applications to girth of Cayley graphs \cite{Gamburd_Hoory_Shahshahani_Shalev_Virag_2009} and mixing times \cite{boyd2009fastest}.

Simple butterfly trees with $N=2^n$ vertices are constructed recursively by applying a sequence of left and right merging operators to two identical copies of the previous-level tree. When the number of left and right merges coincide, the resulting tree spans a square lattice; in general it spans a rectangular lattice (see \cite{PZ25} for a broader theory of block products of trees). The comb graphs arise as extremal instances in which all left merges occur consecutively, followed by all right merges (or vice versa).

\subsection{Burning behavior of butterfly trees}

The next examples in \Cref{fig:8x8_comparison} illustrate four distinct simple butterfly trees spanning the $8 \times 8$ lattice.

The first is the standard comb graph $C_{8,8}$, that can be realized as a simple butterfly tree (with four left-merging operations followed by four right-merging operations). As seen earlier, then $b(C_{8,8}) = 8$. So two additional fires are necessary to fully burn the comb graph versus the integer lattice.

The next two constructions may be viewed as stacked combs. For instance, one may stack two copies of $C_{8,4}$ by identifying an additional boundary edge on the left. Equivalently, this graph can be obtained as the butterfly tree by performing three left merging operations (to create teeth of length four), followed by four right merging operations (to form the spine and hence $C_{8,4}$), and then one additional left merge to stack a second copy of $C_{8,4}$ along the boundary.

A burning sequence for this stacked $C_{8,4}$ configuration can be derived from a successful sequence for $C_{16,4}$ by splitting $C_{16,4}$ into two halves and stacking them vertically. Since
\[
b(C_{16,4}) = 4 - 1 + \left\lceil \sqrt{16 - 4 + 1} \right\rceil
= 3 + \lceil \sqrt{13} \rceil = 7,
\]
a greedy burning sequence for $C_{16,4}$ yields a seven-round burning sequence for the stacked $C_{8,4}$ model. Similarly, as $b(C_{32,2}) = 7$, one may decompose $C_{32,2}$ into four vertically stacked copies of $C_{8,2}$, again obtaining a burning sequence in seven rounds.

In both cases, the additional merging introduces new communication pathways between previously distant portions of the comb, reducing the burning number from eight for $C_{8,8}$ to seven. Moreover, the adapted burning sequences typically require only the earliest fires to cover the entire graph within the allotted rounds.

The final example alternates left and right merges at each stage. (This maximizes the Horton–Strahler (HS) number, a measure of branching complexity; see \cite{peca2025horton}.) This more balanced recursive structure again admits a burning sequence completing in seven rounds.

\subsection{Discussion}

Comb graphs represent extremal, highly unbalanced instances of simple butterfly trees. Introducing even partial alternation of the merging operations, while preserving the same total number of left and right merges required to span the square lattice, strictly reduces the burning number in the $8 \times 8$ case. This suggests that recursive balance and branching complexity influence the efficiency of fire spread. In particular, subcubic spanning trees that better approximate two-dimensional expansion appear to narrow the gap between the burning behavior of the lattice and that of its spanning trees.

These observations motivate a broader study of the burning number across the full family of simple butterfly trees. Since every such tree on $N = 2^n$ vertices is determined by a sequence of left and right merging operations, one may ask how the burning number depends on this sequence. For instance, randomizing the merge pattern yields a natural distribution on degree-constrained spanning trees of the lattice, providing a model for studying typical burning behavior under recursive construction.

Finally, it is worth noting that simple butterfly trees are uniformly degree-bounded by three, as the recursive construction fixes the root at a corner and only extends the top left or right edge of the parent copy to the root of the child copy during each merging operation. Understanding how burning behaves within this restricted class of spanning trees may offer additional insight into the BNC, particularly in light of the NP-completeness of burning even for trees. The butterfly family provides a structured yet flexible testbed for exploring how branching balance, recursive symmetry, and geometric spread interact in determining burning time.

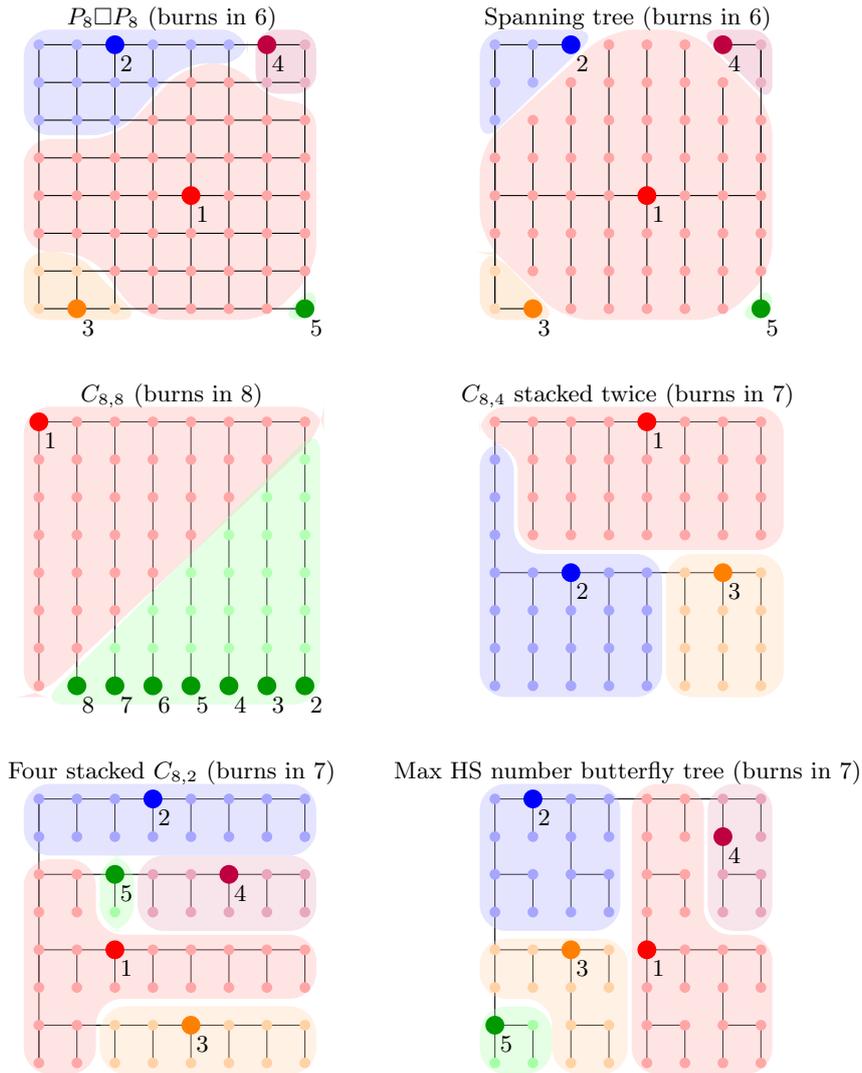
\begin{figure}[t]
\centering
\begin{tikzpicture}[scale=0.5]


\def\panelwidth{12}     
\def\panelheight{10}    
\def\n{8}
\def\m{8}


\newcommand{\DrawFullGrid}{
    \foreach \i in {1,...,\n}{
        \foreach \j in {1,...,\m}{
            \node[circle,fill=black,inner sep=1.2pt] at (\i,-\j) {};
        }
    }
    \foreach \i in {1,...,7}{
        \foreach \j in {1,...,\m}{
            \draw (\i,-\j) -- (\the\numexpr\i+1\relax,-\j);
        }
    }
    \foreach \i in {1,...,\n}{
        \foreach \j in {1,...,7}{
            \draw (\i,-\j) -- (\i,-\the\numexpr\j+1\relax);
        }
    }
}


\begin{scope}[shift={(0,0)}]
\DrawFullGrid

\begin{scope}[on background layer]

\fill[red!30,opacity=0.35,rounded corners=10pt]
    (0.6,-3.5) -- (3,-3.5) -- (4.9,-1.5) -- (6.3,-1.5) -- (7,-2.5) -- (8.3,-2.5) -- (8.3,-7) -- (7,-8.3) -- (3.7,-8.3) -- (2.3,-6.5) -- (0.6,-6.5) --  cycle;

\fill[blue!30,opacity=0.35,rounded corners=10pt]
    (.6,-0.6) -- (6.5,-0.6) -- (6.5,-1.5)-- (4.7,-1.5) -- (3.1,-3.4) -- (.6,-3.4) -- cycle;

\fill[orange!35,opacity=0.4,rounded corners=8pt]
    (.6, -6.5) --  (2,-6.5)  -- (3.7,-8.3)-- (.6,-8.3) --cycle;

\fill[purple!35,opacity=0.4,rounded corners=8pt]
    (6.7, -.6) --  (8.3,-.6)  -- (8.3,-2.3)-- (6.7,-2.3) --cycle;

\fill[green!35,opacity=0.4,rounded corners=8pt]
    (7.3, -8.3) --  (8.3,-7.3)  -- (8.3,-8.3)--cycle;

\end{scope}

\foreach \i in {4,...,7} { \fill[red!35] (\i,-8) circle (4pt); }
\foreach \i in {3,...,8} { \fill[red!35] (\i,-7) circle (4pt); }
\foreach \i in {1,...,8} { \fill[red!35] (\i,-6) circle (4pt); }
\foreach \i in {1,...,8} { \fill[red!35] (\i,-5) circle (4pt); }
\foreach \i in {1,...,8} { \fill[red!35] (\i,-4) circle (4pt); }
\foreach \i in {4,...,8 } { \fill[red!35] (\i,-3) circle (4pt); }
\foreach \i in {5,...,6} { \fill[red!35] (\i,-2) circle (4pt); }
\node[circle,fill=red,inner sep=2.5pt] at (5,-5) {};

\foreach \i in {1,...,6} { \fill[blue!30] (\i,-1) circle (4pt); }
\foreach \i in {1,...,4} { \fill[blue!30] (\i,-2) circle (4pt); }
\foreach \i in {1,...,3} { \fill[blue!30] (\i,-3) circle (4pt); }

\node[circle,fill=blue,inner sep=2.5pt] at (3,-1) {};

\foreach \i in {1,...,3} { \fill[orange!30] (\i,-8) circle (4pt); }
\foreach \i in {1,...,2} { \fill[orange!30] (\i,-7) circle (4pt); }

\node[circle,fill=orange,inner sep=2.5pt] at (2,-8) {};

\foreach \i in {7,...,8} { \fill[purple!30] (\i,-1) circle (4pt); }
\foreach \i in {7,...,8} { \fill[purple!30] (\i,-2) circle (4pt); }

\node[circle,fill=purple,inner sep=2.5pt] at (7,-1) {};

\foreach \i in {8,...,8} { \fill[green!40] (\i,-8) circle (4pt); }
\node[circle,fill=green!60!black,inner sep=2.5pt] at (8,-8) {};

\node at (5.3,-5.5) {\small 1};
\node at (3.3,-1.5) {\small 2};
\node at (2.3,-8.5) {\small 3};
\node at (7.3,-1.5) {\small 4};
\node at (8.3,-8.5) {\small 5};

\node at (4.5,-.3) {\small $P_8 \square P_8$ (burns in 6)};
\end{scope}

\begin{scope}[shift={(\panelwidth,0)}]

\foreach \i in {1,...,\n}{
    \foreach \j in {1,...,\m}{
        \node[circle,fill=black,inner sep=1.2pt] at (\i,-\j) {};
    }
}

\foreach \i in {1,...,\n}{
    \node[circle,fill=black,inner sep=1.2pt] at (\i,-1) {};
}
\foreach \i in {1,...,7}{
    \draw (\i,-5) -- (\the\numexpr\i+1\relax,-5);
}

\foreach \i in {1,...,2}{
    \node[circle,fill=black,inner sep=1.2pt] at (\i,-1) {};
}
\foreach \i in {1,...,2}{
    \draw (\i,-1) -- (\the\numexpr\i+1\relax,-1);
}
\foreach \i in {7,...,7}{
    \node[circle,fill=black,inner sep=1.2pt] at (\i,-1) {};
}
\foreach \i in {7,...,7}{
    \draw (\i,-1) -- (\the\numexpr\i+1\relax,-1);
}

\foreach \i in {1,...,1}{
    \node[circle,fill=black,inner sep=1.2pt] at (\i,-8) {};
}
\foreach \i in {1,...,1}{
    \draw (\i,-8) -- (\the\numexpr\i+1\relax,-8);
}

\foreach \i in {1,4,5,6,\n}{
    \foreach \j in {2,...,\m}{
        \node[circle,fill=black,inner sep=1.2pt] at (\i,-\j) {};
        \draw (\i,-\the\numexpr\j-1\relax) -- (\i,-\j);
    }
}

\foreach \i in {3,7,}{
    \foreach \j in {3,...,\m}{
        \node[circle,fill=black,inner sep=1.2pt] at (\i,-\j) {};
        \draw (\i,-\the\numexpr\j-1\relax) -- (\i,-\j);
    }
}

\foreach \i in {2,8,}{
    \foreach \j in {4,...,7}{
        \node[circle,fill=black,inner sep=1.2pt] at (\i,-\j) {};
        \draw (\i,-\the\numexpr\j-1\relax) -- (\i,-\j);
    }
}

\foreach \i in {2,8,}{
    \foreach \j in {2,...,2}{
        \node[circle,fill=black,inner sep=1.2pt] at (\i,-\j) {};
        \draw (\i,-\the\numexpr\j-1\relax) -- (\i,-\j);
    }
}

\begin{scope}[on background layer]

\fill[red!30,opacity=0.35,rounded corners=10pt]
    (0.6,-3.7) -- (4,-.6) -- (6.2,-.6) -- (8.3,-2.7) -- (8.3,-7) -- (7,-8.3) -- (2.8,-8.3) -- (0.6,-6.1) --  cycle;

\fill[blue!30,opacity=0.35,rounded corners=10pt]
    (.6,-0.6) -- (3.8,-0.6) -- (.6,-3.7) -- cycle;

\fill[orange!35,opacity=0.4,rounded corners=8pt]
    (.6, -6.5) --  (1,-6.5)  -- (2.7,-8.3)-- (.6,-8.3) --cycle;

\fill[purple!35,opacity=0.4,rounded corners=8pt]
    (6.3, -.6) --  (8.3,-.6)  -- (8.3,-2.6) --cycle;

\fill[green!35,opacity=0.4,rounded corners=8pt]
    (7.3, -8.3) --  (8.3,-7.3)  -- (8.3,-8.3)--cycle;

\end{scope}


\foreach \i in {3,...,7} { \fill[red!35] (\i,-8) circle (4pt); }
\foreach \i in {2,...,8} { \fill[red!35] (\i,-7) circle (4pt); }
\foreach \i in {1,...,8} { \fill[red!35] (\i,-6) circle (4pt); }
\foreach \i in {1,...,8} { \fill[red!35] (\i,-5) circle (4pt); }
\foreach \i in {1,...,8} { \fill[red!35] (\i,-4) circle (4pt); }
\foreach \i in {2,...,8 } { \fill[red!35] (\i,-3) circle (4pt); }
\foreach \i in {3,...,7} { \fill[red!35] (\i,-2) circle (4pt); }
\foreach \i in {4,...,6} { \fill[red!35] (\i,-1) circle (4pt); }
\node[circle,fill=red,inner sep=2.5pt] at (5,-5) {};

\foreach \i in {1,...,3} { \fill[blue!30] (\i,-1) circle (4pt); }
\foreach \i in {1,...,2} { \fill[blue!30] (\i,-2) circle (4pt); }
\foreach \i in {1,...,1} { \fill[blue!30] (\i,-3) circle (4pt); }

\node[circle,fill=blue,inner sep=2.5pt] at (3,-1) {};

\foreach \i in {1,...,2} { \fill[orange!30] (\i,-8) circle (4pt); }
\foreach \i in {1,...,1} { \fill[orange!30] (\i,-7) circle (4pt); }

\node[circle,fill=orange,inner sep=2.5pt] at (2,-8) {};

\foreach \i in {7,...,8} { \fill[purple!30] (\i,-1) circle (4pt); }
\foreach \i in {8,...,8} { \fill[purple!30] (\i,-2) circle (4pt); }

\node[circle,fill=purple,inner sep=2.5pt] at (7,-1) {};

\foreach \i in {8,...,8} { \fill[green!40] (\i,-8) circle (4pt); }
\node[circle,fill=green!60!black,inner sep=2.5pt] at (8,-8) {};

\node at (5.3,-5.5) {\small 1};
\node at (3.3,-1.5) {\small 2};
\node at (2.3,-8.5) {\small 3};
\node at (7.3,-1.5) {\small 4};
\node at (8.3,-8.5) {\small 5};

\node at (4.5,-.3) {\small Spanning tree (burns in 6)};
\end{scope}


\begin{scope}[shift={(0,-\panelheight)}]

\foreach \i in {1,...,\n}{
    \node[circle,fill=black,inner sep=1.2pt] at (\i,-1) {};
}
\foreach \i in {1,...,7}{
    \draw (\i,-1) -- (\the\numexpr\i+1\relax,-1);
}

\foreach \i in {1,...,\n}{
    \foreach \j in {2,...,\m}{
        \node[circle,fill=black,inner sep=1.2pt] at (\i,-\j) {};
        \draw (\i,-\the\numexpr\j-1\relax) -- (\i,-\j);
    }
}

\fill[red!30,opacity=0.35,rounded corners=10pt]
    (0.6,-.6) -- (0.6,-8.3) -- (1,-8.3) -- (8.5,-1.2) -- (8.5,-.6)--  cycle;

\fill[green!30,opacity=0.35,rounded corners=10pt]
     (1,-8.5) -- (8.4,-8.5) -- (8.4,-1.1) --  cycle;

\foreach \i in {1,...,8} { \fill[red!35] (\i,-1) circle (4pt); }
\foreach \i in {1,...,7} { \fill[red!35] (\i,-2) circle (4pt); }
\foreach \i in {1,...,6} { \fill[red!35] (\i,-3) circle (4pt); }
\foreach \i in {1,...,5} { \fill[red!35] (\i,-4) circle (4pt); }
\foreach \i in {1,...,4} { \fill[red!35] (\i,-5) circle (4pt); }
\foreach \i in {1,...,3 } { \fill[red!35] (\i,-6) circle (4pt); }
\foreach \i in {1,...,2} { \fill[red!35] (\i,-7) circle (4pt); }
\foreach \i in {1,...,1} { \fill[red!35] (\i,-8) circle (4pt); }
\node[circle,fill=red,inner sep=2.5pt] at (1,-1) {};

\foreach \i in {2,...,8} { \fill[green!30] (\i,-8) circle (4pt); }
\foreach \i in {3,...,8} { \fill[green!30] (\i,-7) circle (4pt); }
\foreach \i in {4,...,8} { \fill[green!30] (\i,-6) circle (4pt); }
\foreach \i in {5,...,8} { \fill[green!30] (\i,-5) circle (4pt); }
\foreach \i in {6,...,8} { \fill[green!30] (\i,-4) circle (4pt); }
\foreach \i in {7,...,8} { \fill[green!30] (\i,-3) circle (4pt); }
\foreach \i in {8,...,8} { \fill[green!30] (\i,-2) circle (4pt); }

\node[circle,fill=green!60!black,inner sep=2.5pt] at (8,-8) {};
\node[circle,fill=green!60!black,inner sep=2.5pt] at (7,-8) {};
\node[circle,fill=green!60!black,inner sep=2.5pt] at (6,-8) {};
\node[circle,fill=green!60!black,inner sep=2.5pt] at (5,-8) {};
\node[circle,fill=green!60!black,inner sep=2.5pt] at (4,-8) {};
\node[circle,fill=green!60!black,inner sep=2.5pt] at (3,-8) {};
\node[circle,fill=green!60!black,inner sep=2.5pt] at (2,-8) {};

\node at (1.3,-1.5) {\small 1};
\node at (8.3,-8.5) {\small 2};
\node at (7.3,-8.5) {\small 3};
\node at (6.3,-8.5) {\small 4};
\node at (5.3,-8.5) {\small 5};
\node at (4.3,-8.5) {\small 6};
\node at (3.3,-8.5) {\small 7};
\node at (2.3,-8.5) {\small 8};

\node at (4.5,-.3) {\small $C_{8,8}$ (burns in 8)};
\end{scope}

\begin{scope}[shift={(\panelwidth,-\panelheight)}]

\foreach \i in {1,...,\n}{
    \node[circle,fill=black,inner sep=1.2pt] at (\i,-5) {};
}
\foreach \i in {1,...,7}{
    \draw (\i,-5) -- (\the\numexpr\i+1\relax,-5);
}

\foreach \i in {1,...,\n}{
    \foreach \j in {6,7,8}{
        \node[circle,fill=black,inner sep=1.2pt] at (\i,-\j) {};
        \draw (\i,-\the\numexpr\j-1\relax) -- (\i,-\j);
    }
}

\foreach \i in {1,...,\n}{
    \node[circle,fill=black,inner sep=1.2pt] at (\i,-1) {};
}
\foreach \i in {1,...,7}{
    \draw (\i,-1) -- (\the\numexpr\i+1\relax,-1);
}

\foreach \i in {1,...,\n}{
    \foreach \j in {2,3,4}{
        \node[circle,fill=black,inner sep=1.2pt] at (\i,-\j) {};
        \draw (\i,-\the\numexpr\j-1\relax) -- (\i,-\j);
    }
}

\draw (1,-4) -- (1,-5);

\fill[red!30,opacity=0.35,rounded corners=10pt]
    (0.6,-.6) -- (8.6,-.6) -- (8.6,-4.4) -- (1.6,-4.4) -- (1.6,-1.6) -- (0.6,-1.6) --  cycle;

\fill[blue!30,opacity=0.35,rounded corners=10pt]
     (.6,-1.5) -- (1.5,-1.6) -- (1.5,-4.5) -- (5.4,-4.5) -- (5.4,-8.3) -- (.6,-8.3)  --  cycle;

\fill[orange!30,opacity=0.35,rounded corners=10pt]
     (5.5,-4.5) -- (8.6,-4.5) -- (8.6,-8.3) -- (5.5,-8.3)  --  cycle;

\foreach \i in {1,...,8} { \fill[red!35] (\i,-1) circle (4pt); }
\foreach \i in {2,...,8} { \fill[red!35] (\i,-2) circle (4pt); }
\foreach \i in {2,...,8} { \fill[red!35] (\i,-3) circle (4pt); }
\foreach \i in {2,...,8} { \fill[red!35] (\i,-4) circle (4pt); }

\node[circle,fill=red,inner sep=2.5pt] at (5,-1) {};

\foreach \i in {2,...,8} { \fill[blue!35] (1,-\i) circle (4pt); }
\foreach \i in {5,...,8} { \fill[blue!35] (2,-\i) circle (4pt); }
\foreach \i in {5,...,8} { \fill[blue!35] (3,-\i) circle (4pt); }
\foreach \i in {5,...,8} { \fill[blue!35] (4,-\i) circle (4pt); }
\foreach \i in {5,...,8} { \fill[blue!35] (5,-\i) circle (4pt); }

\node[circle,fill=blue,inner sep=2.5pt] at (3,-5) {};

\foreach \i in {5,...,8} { \fill[orange!35] (6,-\i) circle (4pt); }
\foreach \i in {5,...,8} { \fill[orange!35] (7,-\i) circle (4pt); }
\foreach \i in {5,...,8} { \fill[orange!35] (8,-\i) circle (4pt); }

\node[circle,fill=orange,inner sep=2.5pt] at (7,-5) {};

\node at (5.3,-1.5) {\small 1};
\node at (3.3,-5.5) {\small 2};
\node at (7.3,-5.5) {\small 3};

\node at (4.5,-.3) {\small $C_{8,4}$ stacked twice (burns in 7)};
\end{scope}


\begin{scope}[shift={(0,-2*\panelheight)}]

\foreach \block in {0,2,4,6}{
    
    \foreach \i in {1,...,\n}{
        \node[circle,fill=black,inner sep=1.2pt] at (\i,-\the\numexpr 1+\block\relax) {};
    }
    \foreach \i in {1,...,7}{
        \draw (\i,-\the\numexpr 1+\block\relax) 
            -- (\the\numexpr\i+1\relax,-\the\numexpr 1+\block\relax);
    }

    \foreach \i in {1,...,\n}{
        \node[circle,fill=black,inner sep=1.2pt] 
            at (\i,-\the\numexpr 2+\block\relax) {};
        \draw (\i,-\the\numexpr 1+\block\relax) 
            -- (\i,-\the\numexpr 2+\block\relax);
    }
}

\foreach \j in {1,...,7}{
    \draw (1,-\j) -- (1,-\the\numexpr\j+1\relax);
}

\fill[red!30,opacity=0.35,rounded corners=10pt]
    (0.6,-2.6) -- (2.5,-2.6) -- (2.5,-4.6) -- (8.3,-4.6) -- (8.3,-6.3) -- (2.5,-6.3) -- (2.5, -8.3) -- (0.6, -8.3) -- cycle;

\fill[blue!30,opacity=0.35,rounded corners=10pt]
     (.6,-.6) -- (8.3,-.6) -- (8.3,-2.5) -- (.6,-2.5) --  cycle;

\fill[orange!30,opacity=0.35,rounded corners=10pt]
     (2.6,-6.5) -- (8.3,-6.5) -- (8.3,-8.3) -- (2.6,-8.3)  --  cycle;

\fill[purple!30,opacity=0.35,rounded corners=10pt]
     (3.6,-2.5) -- (8.3,-2.5) -- (8.3,-4.5) -- (3.6,-4.5)  --  cycle;

\fill[green!30,opacity=0.35,rounded corners=10pt]
     (2.6,-2.5) -- (3.5,-2.5) -- (3.5,-4.5) -- (2.6,-4.5)  --  cycle;

\foreach \i in {3,...,8} { \fill[red!35] (1,-\i) circle (4pt); }
\foreach \i in {3,...,8} { \fill[red!35] (2,-\i) circle (4pt); }
\foreach \i in {3,...,8} { \fill[red!35] (\i,-5) circle (4pt); }
\foreach \i in {3,...,8} { \fill[red!35] (\i,-6) circle (4pt); }

\node[circle,fill=red,inner sep=2.5pt] at (3,-5) {};

\foreach \i in {1,...,8} { \fill[blue!35] (\i,-1) circle (4pt); }
\foreach \i in {1,...,8} { \fill[blue!35] (\i,-2) circle (4pt); }

\node[circle,fill=blue,inner sep=2.5pt] at (4,-1) {};

\foreach \i in {3,...,8} { \fill[orange!35] (\i,-7) circle (4pt); }
\foreach \i in {3,...,8} { \fill[orange!35] (\i,-8) circle (4pt); }

\node[circle,fill=orange,inner sep=2.5pt] at (5,-7) {};

\foreach \i in {4,...,8} { \fill[purple!35] (\i,-3) circle (4pt); }
\foreach \i in {4,...,8} { \fill[purple!35] (\i,-4) circle (4pt); }

\node[circle,fill=purple,inner sep=2.5pt] at (6,-3) {};

\foreach \i in {3,...,4} { \fill[green!35] (3,-\i) circle (4pt); }

\node[circle,fill=green!60!black,inner sep=2.5pt] at (3,-3) {};

\node at (3.3,-5.5) {\small 1};
\node at (4.3,-1.5) {\small 2};
\node at (5.3,-7.5) {\small 3};
\node at (6.3,-3.5) {\small 4};
\node at (3.3,-3.5) {\small 5};

\node at (4.5,-.3) {\small Four stacked $C_{8,2}$ (burns in 7)};
\end{scope}

\begin{scope}[shift={(\panelwidth,-2*\panelheight)}]


\foreach \sx/\sy in {0/0,4/0,0/4,4/4}{
  \begin{scope}[shift={(\sx,-\sy)}]

    \foreach \x in {1,...,4}{
      \foreach \y in {1,...,4}{
        \node[circle,fill=black,inner sep=1.2pt] (v\x\y) at (\x,-\y) {};
      }
    }

    \foreach \x in {1,...,3}{
      \draw (\x,-1) -- (\the\numexpr\x+1\relax,-1);
    }

    \foreach \y in {1,...,3}{
      \draw (1,-\y) -- (1,-\the\numexpr\y+1\relax);
    }

    \draw (3,-1) -- (3,-2) -- (3,-3) -- (3,-4);

    \draw (2,-1) -- (2,-2);
    \draw (4,-1) -- (4,-2);
    \draw (2,-3) -- (2,-4);
    \draw (4,-3) -- (4,-4);

    \draw (1,-3) -- (2,-3);
    \draw (3,-3) -- (4,-3);
    
    \end{scope}
}


\draw (4,-1) -- (5,-1);

\draw (1,-4) -- (1,-5);

\draw (5,-4) -- (5,-5);

\fill[red!30,opacity=0.35,rounded corners=10pt]
    (4.6,-.6) -- (6.5,-.6) -- (6.5,-4.6) -- (8.3,-4.6) -- (8.3,-8.3) -- (4.6,-8.3) -- cycle;

\fill[blue!30,opacity=0.35,rounded corners=10pt]
     (.6,-.6) -- (4.3,-.6) -- (4.3,-4.5) -- (.6,-4.5) --  cycle;

\fill[orange!30,opacity=0.35,rounded corners=10pt]
     (.6,-4.7) -- (4.5,-4.7) -- (4.5,-8.3) -- (2.5,-8.3)  -- (2.5,-6.3) -- (.6,-6.3) --  cycle;

\fill[purple!30,opacity=0.35,rounded corners=10pt]
     (6.6,-.6) -- (8.3,-.6) -- (8.3,-4.5) -- (6.6,-4.5)  --  cycle;

\fill[green!30,opacity=0.35,rounded corners=10pt]
     (.6,-6.5) -- (2.5,-6.5) -- (2.5,-8.3) -- (.6,-8.3)  --  cycle;

\foreach \i in {1,...,8} { \fill[red!35] (5,-\i) circle (4pt); }
\foreach \i in {1,...,8} { \fill[red!35] (6,-\i) circle (4pt); }
\foreach \i in {5,...,8} { \fill[red!35] (7,-\i) circle (4pt); }
\foreach \i in {5,...,8} { \fill[red!35] (8,-\i) circle (4pt); }

\node[circle,fill=red,inner sep=2.5pt] at (5,-5) {};

\foreach \i in {1,...,4} { \fill[blue!35] (\i,-1) circle (4pt); }
\foreach \i in {1,...,4} { \fill[blue!35] (\i,-2) circle (4pt); }
\foreach \i in {1,...,4} { \fill[blue!35] (\i,-3) circle (4pt); }
\foreach \i in {1,...,4} { \fill[blue!35] (\i,-4) circle (4pt); }

\node[circle,fill=blue,inner sep=2.5pt] at (2,-1) {};

\foreach \i in {1,...,4} { \fill[orange!35] (\i,-5) circle (4pt); }
\foreach \i in {1,...,4} { \fill[orange!35] (\i,-6) circle (4pt); }
\foreach \i in {3,...,4} { \fill[orange!35] (\i,-7) circle (4pt); }
\foreach \i in {3,...,4} { \fill[orange!35] (\i,-8) circle (4pt); }

\node[circle,fill=orange,inner sep=2.5pt] at (3,-5) {};

\foreach \i in {1,...,4} { \fill[purple!35] (7,-\i) circle (4pt); }
\foreach \i in {1,...,4} { \fill[purple!35] (8,-\i) circle (4pt); }

\node[circle,fill=purple,inner sep=2.5pt] at (7,-2) {};

\foreach \i in {1,...,2} { \fill[green!35] (\i,-7) circle (4pt); }
\foreach \i in {1,...,2} { \fill[green!35] (\i,-8) circle (4pt); }

\node[circle,fill=green!60!black,inner sep=2.5pt] at (1,-7) {};

\node at (5.3,-5.5) {\small 1};
\node at (2.3,-1.5) {\small 2};
\node at (3.3,-5.5) {\small 3};
\node at (7.3,-2.5) {\small 4};
\node at (1.3,-7.5) {\small 5};

\node at (4.5,-.3) {\small Max HS number butterfly tree (burns in 7)};
\end{scope}

\end{tikzpicture}

\caption{Burning sequences on the $8\times 8$ lattice and selected spanning trees. 
The lattice and a suitably balanced spanning tree burn in six rounds. 
The comb requires eight, while structured stacked constructions reduce this to seven. 
This illustrates the structural sensitivity of the burning number to hierarchical bottlenecks.}
\label{fig:8x8_comparison}
\end{figure}

\end{document}